\documentclass[12pt,a4paper,psamsfonts]{amsart}

\usepackage{fancyhdr}
\usepackage{appendix}
\usepackage{amssymb,amscd,amsxtra,calc}
\usepackage{mathrsfs}
\usepackage{cmmib57}
\usepackage{multirow}
\usepackage[all]{xy}
\usepackage{longtable}
\usepackage[colorlinks,linkcolor=blue,anchorcolor=blue,citecolor=green]{hyperref}
\usepackage{tikz}

\theoremstyle{plain}
    \newtheorem{thm}{Theorem}[section]
    
    \newtheorem{claim}[thm]{Claim}
     \newtheorem{conjecture}[thm]{Conjecture}
    \newtheorem{corollary}[thm]{Corollary}
    
    \newtheorem{lemma}[thm]{Lemma}
    \newtheorem{proposition}[thm]{Proposition}
    \newtheorem{question}[thm]{Question}
    \newtheorem{theorem}[thm]{Theorem}

\theoremstyle{definition}

    \newtheorem{notation}[thm]{Notation}
    \newtheorem*{notation*}{Notation and Terminology}
      \newtheorem*{mainthm}{Main Theorem}
    \newtheorem{remark}[thm]{Remark}
    \newtheorem*{ack}{Acknowledgments}
\theoremstyle{remark}

\newcommand{\C}{\mathbb{C}}
\newcommand{\F}{\mathbb{F}}

\newcommand{\PP}{\mathbb{P}}

\newcommand{\Q}{\mathbb{Q}}

\newcommand{\R}{\mathbb{R}}

\newcommand{\Z}{\mathbb{Z}}

\newcommand{\OO}{\mathcal{O}}

\newcommand{\Amp}{\operatorname{Amp}}

\newcommand{\Aut}{\operatorname{Aut}}

\newcommand{\diag}{\operatorname{diag}}

\newcommand{\id}{\operatorname{id}}

\newcommand{\NE}{\overline{\operatorname{NE}}}
\newcommand{\Nef}{\operatorname{Nef}}

\newcommand{\NS}{\operatorname{NS}}

\newcommand{\Supp}{\operatorname{Supp}}

\newcommand{\N}{\operatorname{N}}

\newcommand{\Pic}{\operatorname{Pic}}

\newcommand{\mstriangle}[1]{
\begin{tikzpicture}[x=0.3cm,y=0.3cm]
\draw (-0.4,-0.433) -- (1.4,-0.433);
\draw (-0.2,-0.7794) -- (0.7,0.7794);
\draw (1.2,-0.7794) -- (0.3,0.7794);
\end{tikzpicture}
}
\newcommand{\mssharp}[1]{
\begin{tikzpicture}[x=0.3cm,y=0.3cm]
\draw (-0.8,-0.5) -- (0.8,-0.5);
\draw (-0.8,0.5) -- (0.8,0.5);
\draw (-0.5,-0.8) -- (-0.5,0.8);
\draw (0.5,-0.8) -- (0.5,0.8);
\end{tikzpicture}
}

\makeatletter

\newcommand{\Rmnum}[1]{\expandafter\@slowromancap\romannumeral #1@}
\makeatother

\begin{document}

\title[Endomorphism of Fano threefold]
{Non-isomorphic endomorphisms of Fano threefolds}

\author{Sheng Meng, De-Qi Zhang, Guolei Zhong}

\address{Korea Institute For Advanced Study,
Seoul 02455, Republic of Korea}
\email{ms@u.nus.edu, shengmeng@kias.re.kr}

\address
{
\textsc{National University of Singapore,
Singapore 119076, Republic of Singapore
}}
\email{matzdq@nus.edu.sg}

\address
{
\textsc{National University of Singapore,
Singapore 119076, Republic of Singapore
}}
\email{zhongguolei@u.nus.edu}
\begin{abstract}
Let $X$ be a smooth Fano threefold.
We show that $X$ admits a non-isomorphic surjective endomorphism if and only if
$X$ is either a toric variety or a product of $\mathbb{P}^1$ and a del Pezzo surface; in this case, $X$
is a rational variety. We further show that $X$ admits a polarized (or amplified) endomorphism if and only if $X$ is a toric variety.
\end{abstract}



\maketitle
\tableofcontents

\section{Introduction}

We work over an algebraically closed field $k$ of characteristic $0$.

It has been a long history and involved many people working on the classification of smooth projective varieties $X$ admitting a non-isomorphic surjective endomorphism $f$.
When $\dim (X) = 1$, by the Hurwitz formula, such $X$ is  a rational or an elliptic curve.
When $\dim (X) = 2$, such (possibly singular) $X$  has been fully studied by Nakayama (cf.~\cite{Nak08}).

In higher dimensions, in view of the endomorphism-descending (or lifting) property of the three typical fibrations
(or Beauville-Bogomolov decomposition): Albanese map, Kodaira fibration, and maximal rationally connected fibration, rationally connected varieties
are the essential cases in studying
non-isomorphic surjective endomorphisms on them.

Further, to remove the non-essential case of the product of a non-isomorphic surjective endomorphism and an automorphism, we need some constraint
on $f$.
The right constraint seems being that $f : X \to X$ is {\it int-amplified}, i.e., $f^*L-L$ is ample for some ample Cartier divisor $L$ on $X$,
or equivalently, every eigenvalue of $f^*|_{\NS(X)\otimes_{\Z} \C}$ has modulus $>1$
(cf.~Lemma \ref{lem-int-amp-equiv}).
Precisely, generalizing \cite[Question 4.4]{Fak03}, the first and third authors have asked the following Question \ref{main-que-toric} (cf.~\cite[Question 1.2]{MZg20}).
Recall that $f$ is {\it polarized} if $f^*H \sim qH$ for some ample Cartier divisor $H$ and integer $q > 1$
(cf. Lemma \ref{lem-pol-mod>1}).
Note that every polarized endomorphism is int-amplified.

\begin{question}\label{main-que-toric}
Let $X$ be a rationally connected smooth projective variety. Suppose that $X$ admits an int-amplified (or polarized) endomorphism $f$. Is $X$ a toric variety?
\end{question}

Question \ref{main-que-toric} itself generalizes the following long-standing Conjecture \ref{main-conj-pn}
of the 1980's.
It has been proved for homogeneous spaces (cf.~\cite{PS89}), for Fano threefolds of Picard number one (cf.~\cite{ARV99},
\cite{HM03}), and for
hypersurfaces of the projective space (cf.~\cite[Proposition 8]{PS89}, \cite{Bea01}).
Recall that a normal projective variety $X$ is {\it Fano} if the anti-canonical divisor $-K_X$ is an  ample $\Q$-Cartier divisor.

\begin{conjecture}\label{main-conj-pn}
Let $X$ be a smooth Fano variety of Picard number one. Suppose that $X$ admits a non-isomorphic surjective endomorphism $f$.
Then $X$ is a projective space.
\end{conjecture}

In this paper, we shall give a positive answer to Question \ref{main-que-toric} for smooth Fano threefolds (cf.~Theorem \ref{main-thm-toric}).
More generally, in Main Theorem below, we give the criterion for the existence of a non-isomorphic surjective endomorphism.

Now we state our main results.

\begin{mainthm}\label{main-thm-niso}
A smooth Fano threefold
is either toric or a product of $\mathbb{P}^1$ and a del Pezzo surface
if and only if
it admits a non-isomorphic surjective endomorphism.
\end{mainthm}

The Main Theorem follows immediately from Theorems \ref{main-thm-prod} and \ref{main-thm-toric} below.

Recall that a surjective endomorphism $f: X \to X$
is {\it amplified} if $f^*L - L$ is ample for some
(not necessarily ample) Cartier divisor $L$.
Note that every int-amplified endomorphism is amplified and the converse holds true if $X$ is Fano (cf. Lemma \ref{lem-MDS-amp}).

\begin{theorem}\label{main-thm-prod}
Let $X$ be a smooth Fano threefold. Then the following are equivalent.
\begin{enumerate}
\item[(1)]
$X$ is a product of $\mathbb{P}^1$ and a del Pezzo surface.
\item[(2)]
$X$ admits a non-isomorphic surjective endomorphism which is not polarized even after iteration.
\item[(3)]
$X$ admits a non-isomorphic surjective endomorphism which is non-amplified (or equivalently, non-int-amplified) (cf. Lemma \ref{lem-MDS-amp}).
\end{enumerate}
\end{theorem}

\begin{theorem}\label{main-thm-toric}
Let $X$ be a smooth Fano threefold. Then the following are equivalent.
\begin{enumerate}
\item[(1)]
$X$ is a toric variety.
\item[(2)]
$X$ admits a polarized endomorphism.
\item[(3)]
$X$ admits an amplified
(or equivalently, int-amplified)
endomorphism (cf. Lemma \ref{lem-MDS-amp}).
\end{enumerate}
\end{theorem}

Question \ref{main-que-toric} is known to be true under further assumption that $f$ has totally invariant ramification (cf.~\cite[Theorem 1.4]{MZg20}, \cite[Corollary 1.4]{MZ19a} and \cite[Theorem 1.2]{HN11}).
However, even when $X$ is toric, $f^{-1}$ may not fix its big torus, e.g. take a general endomorphism of $\PP^n$. This actually makes the proving of our results more difficult.

Fano threefolds have no ``primitive" endomorphisms unless they are toric. Indeed:

\begin{corollary}\label{cor-2}
Let $X$ be a smooth Fano threefold, and $f: X \to X$ a non-isomorphic surjective endomorphism.
Then either $X$ is toric; or $X = \PP^1 \times S$ where $S$ is a del Pezzo surface and (after iteration)
$f = f_P \times f_S$ with $f_P : \PP^1 \to \PP^1$ and $f_S \in \Aut(S)$.
\end{corollary}

Del Pezzo surfaces being rational,
our Main Theorem recovers
\cite[Theorem 1.2]{Zha12}:

\begin{corollary}
A smooth Fano threefold admitting a non-isomorphic surjective endomorphism is rational.
\end{corollary}

\begin{remark}[\textbf{Classification-free approach}]
In the first trial of proving our theorems,
the authors utilized the
classification of smooth Fano
varieties which is available in dimension $\le 3$
but is not plausible in higher dimension;
it needs tedious and patient case by case checking and it would actually be 10-page longer than the present (second)
classification-free approach; this second approach adopted here might be useful
in solving Questions and Conjectures: \ref{main-que-toric}, \ref{main-conj-pn}, \ref{main-que-ti} and
\ref{main-conj-linear} in higher dimensions.
\end{remark}

Besides the enlightening structural propositions in \cite{MM83}, our key ingredient of proving Theorems \ref{main-thm-prod} and \ref{main-thm-toric}
(both for ``(3) $\Rightarrow$ (1)"), without resorting to the detailed classification of Fano threefolds, is to show that a fibre-dimension one Fano contraction (if exists) $\tau:X\to Y$ is eventually a splitting $\mathbb{P}^1$-bundle over a toric surface by the following steps, where the second and last implications require the dynamical assumption:
$$\tau \Rightarrow \text{conic bundle} \Rightarrow \mathbb{P}^1\text{-bundle} \Rightarrow \text{algebraic }\mathbb{P}^1\text{-bundle } \Rightarrow \text{ splitting } \mathbb{P}^1\text{-bundle}.$$
The first three implications are proved in Theorem \ref{thm-conic}.
The last implication is discussed in Theorems \ref{thm-amerik2} and \ref{thm-conic-split-p1p1-blp2} even for not necessarily Fano $X$, which is crucial in Step B below.
Note that a splitting $\mathbb{P}^1$-bundle over a toric variety is toric (cf.~Proposition \ref{prop-toric-pair}).

In the situation of Theorem \ref{main-thm-toric}
(for ``(3) $\Rightarrow$ (1)"),
we further show:

\begin{itemize}
\item[(\textbf{Step A})]
If $\rho(X)\le 2$, then $X$ is either $\mathbb{P}^3$, a (toric) splitting $\mathbb{P}^1$-bundle over $\mathbb{P}^2$, or a (toric) blowup of $\mathbb{P}^3$ along a line (cf.~Theorem \ref{thm-rho2}).
\item[(\textbf{Step B})]
If $\rho(X)\ge 3$, then $X$ is a (toric) blowup of a (not necessarily Fano) splitting $\mathbb{P}^1$-bundle $X'$ over a toric surface along disjoint curves which are intersections of some $(f|_{X'})^{-1}$-invariant (after iteration) prime divisors (cf.~Proof of Theorem \ref{thm-rho>=3}).
\end{itemize}

The proof of {\bf Step B} heavily relies on the study of the following Question \ref{main-que-ti} on $f^{-1}$-periodic subvarieties which generalizes another long-standing Conjecture \ref{main-conj-linear} below.

\begin{question}\label{main-que-ti}
Let $f$ be an int-amplified endomorphism of a rationally connected (or rational, or toric) smooth projective variety $X$.
Let $\Sigma$ be the union of $f^{-1}$-periodic (closed) subvarieties.
Is $X$ toric, and if so, is there a toric pair $(X, \Delta)$ such that $\Sigma\subseteq \Delta$?
\end{question}

We will study Question \ref{main-que-ti} in the divisor case with main results Theorems \ref{thm-surface-toricpair} and \ref{thm-boundary-p1bundle},
confirming the surface case and the case of splitting $\PP^1$-bundles over rational surfaces.

For the surface case,
Nakayama confirmed Sato's conjecture that a smooth projective rational surface $Y$ admitting a non-isomorphic surjective endomorphism $g$ is toric (cf.~\cite{Nak02}).
For our purpose of studying a conic bundle $X$ over a toric surface $Y$, i.e., for {\bf Step B} above, this is not enough; we need the info on the $(f|_Y)^{-1}$-periodic curves as in Question \ref{main-que-ti}.
Such info is completed, in view of Lemmas \ref{lem-noniso-nonpol},
 \ref{lem-p1p1-ticurve} and Theorem \ref{thm-surface-toricpair}.

Conjecture \ref{main-conj-linear} below appears in the  study of dynamical systems.
We refer to Theorem \ref{thm-linear-subspace} for the known cases of Conjecture \ref{main-conj-linear} by the works \cite{PS89}, \cite{CL00}, \cite{Bea01}, \cite{Gur03}, \cite[Section 5]{NZ10} and \cite{Hor17}.

\begin{conjecture}\label{main-conj-linear}
Let $f$ be an endomorphism of $X := \PP^n$ with $\deg(f) \ge 2$.
Then any $f^{-1}$-invariant (or $f^{-1}$-periodic) closed subvariety $Y$ is a linear subspace of $X$.
\end{conjecture}

\section{Preliminaries}\label{section-2-pre}

We use the following notation throughout this paper.
\begin{notation}\label{notation2.1}
Let $X$ be a projective variety.

\begin{itemize}
\item The symbols $\sim$ (resp.~$\sim_{\mathbb Q}$, $\equiv$)  denote
the \textit{linear equivalence} (resp.~\textit{$\mathbb Q$-linear equivalence}, \textit{numerical equivalence}) on $\Q$- (or $\R$-) Cartier divisors.
We also use $\equiv$ to denote the {\it numerical equivalence} of $1$-cycles on $X$.

\item Denote by $\textup{NS}(X) = \Pic(X)/\Pic^0(X)$  the {\it N\'eron-Severi group} of $X$.
Let
 $\N^1(X):=\NS(X)\otimes_\mathbb{Z}\mathbb{R}$ the space of $\mathbb{R}$-Cartier divisors modulo  numerical equivalence and $\rho(X) :=\dim_{\mathbb{R}}\N^1(X)$ the
{\it Picard number} of $X$.
Let $\N_1(X)$ be the dual space of $\N^1(X)$ consisting of 1-cycles.
Denote by $\Nef(X)$ (resp. $\Amp(X)$) the cone of {\it nef divisors} (resp. {\it ample divisors}) in $\N^1(X)$ and $\NE(X)$ the dual cone consisting of {\it pseudo-effective 1-cycles} in $\N_1(X)$.

\item Let $f:X\to X$ be a surjective endomorphism. A subset $Y\subseteq X$ is {\it $f^{-1}$-invariant} (resp.~{\it $f^{-1}$-periodic})  if $f^{-1}(Y)=Y$ (resp.~$f^{-s}(Y)=Y$ for some  $s>0$).

\item A surjective endomorphism $f:X\to X$ is \textit{$q$-polarized} if $f^*H\sim qH$ for some ample Cartier divisor $H$ and integer $q>1$; see Lemma \ref{lem-pol-mod>1} for the equivalent definitions.

\item
A surjective endomorphism $f:X\to X$ is {\it amplified} if $f^*L - L$ is ample for some Cartier divisor $L$.
We say that  $f$ is \textit{int-amplified} if $f^*L-L$ is ample for some  \textit{ample} Cartier divisor $L$;
see Lemma \ref{lem-int-amp-equiv} for the equivalent definitions.

Note that every polarized endomorphism is int-amplified and every int-amplified endomorphism is amplified.

\item A smooth projective variety $X$ is {\it rationally connected} if any two general points of $X$ can be connected by a chain of rational curves. 

\item A normal projective variety $X$ is of \textit{Fano type}, if there is an effective Weil $\mathbb{Q}$-divisor $\Delta$ on $X$ such that the pair $(X,\Delta)$ has at worst klt singularities and $-(K_X+\Delta)$ is ample and $\mathbb{Q}$-Cartier.
If $\Delta=0$, we say that $X$ is a  \textit{(klt) Fano variety}.
A smooth Fano surface is usually called a \textit{del Pezzo surface}.

\item Let $X$ be a smooth Fano threefold. We say that $X$ is \textit{imprimitive} if it is isomorphic to the blowup of a smooth Fano threefold $Y$ along a smooth irreducible curve.
We say that $X$ is \textit{primitive} if it is not imprimitive (cf.~\cite[Definition 1.3]{MM83}).

\item A normal variety $X$ of dimension $n$ is a \textit{toric variety} if $X$ contains a {\it big torus} $T=(k^*)^n$ as an (affine) open dense subset such that the natural multiplication action of $T$ on itself extends to an action on the whole variety. In this case, let $B:=X\backslash T$, which is a divisor; the pair $(X,B)$ is said to be a \textit{toric pair}.

It is known that the reflexive sheaf of logarithmic $1$-form $\widehat{\Omega}_X^1(\textup{log}\,B)\cong {\OO}_X^{\oplus n}$ (cf.~\cite[Remark 4.6]{MZ19a} and \cite[Section 4.3, page 87]{Ful93}) and $K_X+B\sim 0$.
Hence, if $X$ is further smooth projective,
then $B$ has exactly $\rho(X) + \dim (X)$
irreducible components (see e.g. the inequality and  cohomology exact sequence in \cite[Theorem 4.5]{MZ19a} and its proof).

\item Let $\tau:X\to Y$ be the blowup of a smooth toric variety $Y$ along a smooth closed subvariety $C$.
We say that $\tau$ is a \textit{toric blowup} if there exists some big torus $T$ acting on $Y$ with $T(C)=C$.
In this case, $X$ is still toric.

\item A fibration $\tau:X\to Y$ of smooth projective varieties is a \textit{conic bundle} if every fibre is isomorphic to a conic, i.e., a scheme of zeros of a nonzero homogeneous form of degree $2$ on $\mathbb{P}^2$.

Denote by $\Delta_\tau:=\{y\in Y~|~\tau~\textup{is not smooth over }y\}$
the \textit{discriminant} of $\tau$.
When $Y$ is a surface, either $\Delta_\tau=\emptyset$ or $\Delta_\tau$ is a (not necessarily irreducible) curve with only ordinary double points (cf.~\cite[Section 6]{MM83}).
If $X$ is further assumed to be Fano, then we say $\tau$ is a \textit{Fano conic bundle}.

\item Let $\tau:X\to Y$ be a fibration of smooth projective varieties.
A subvariety $S\subseteq X$ is said to be a \textit{(cross-)section} of $\tau$ if the restriction $\tau|_S:S\cong Y$ is an isomorphism onto $Y$.
We say that $\tau$ is  a (smooth) \textit{$\mathbb{P}^1$-bundle},
if $\tau$ is smooth and every fibre is isomorphic to $\PP^1$.
We say that $\tau$ is an \textit{algebraic $\mathbb{P}^1$-bundle}, if $\tau$ is a $\mathbb{P}^1$-bundle and $X\cong\mathbb{P}_Y(\mathcal{E})$ for some locally free rank-two sheaf $\mathcal{E}$  on $Y$.
An algebraic $\mathbb{P}^1$-bundle $X:=\mathbb{P}_Y(\mathcal{E})\xrightarrow{\tau} Y$ is said to be a \textit{splitting $\mathbb{P}^1$-bundle} if $\mathcal{E}$ is a direct sum of two invertible sheaves.
Note that a splitting $\mathbb{P}^1$-bundle over a toric normal projective variety is toric (cf.~Proposition \ref{prop-toric-pair}).	
\end{itemize}
\end{notation}

For the convenience of us and readers,
we now recall several results to be used in the subsequent sections.
Lemmas \ref{lem-equiv-MMP} $\sim$ \ref{lem-MDS-amp} supplement Notation \ref{notation2.1} on endomorphisms.

\begin{lemma}\label{lem-equiv-MMP}
Let $f$ be a surjective endomorphism of a normal projective variety $X$.
Then any finite sequence of minimal model program starting from $X$, is $f$-equivariant (after iterating $f$), if one of the following conditions is satisfied.

\begin{enumerate}
\item
The (closed) Mori cone $\NE(X)$ of pseudo-effective $1$-cycles has only finitely many extremal rays
(this holds when $X$ is of Fano type) (cf.~\cite[Lemma 2.11]{Zha10} and \cite[Theorem 3.7]{KM98}).
\item  $\deg (f) \ge 2$ and $\dim (X)=2$  (then $X$ is known to have only lc singularities) (cf.~\cite[Proposition 11]{Nak02} and \cite[Theorem 4.7]{MZ19b}).
\item $X$ admits an int-amplified endomorphism (cf.~\cite[Theorem 1.1]{MZ20}).
\end{enumerate}
\end{lemma}

The following generalizes \cite[Lemma 2.9]{Zha10} to the int-amplified case.

\begin{lemma}[{cf.~\cite[Lemma 2.9]{Zha10}}]\label{lem-non-pe-ti}
Let $f:X\to X$ be a surjective endomorphism of a    $\mathbb{Q}$-factorial normal projective variety of dimension $n$.
Let $S(X)$ be the set of prime divisors $D$ on $X$ such that $D|_D$ is not pseudo-effective.
Then $f^{-1}(S(X))=S(X)$.
If $f$ is further assumed to be int-amplified, then $S(X)$ is a finite set; hence after iteration, $f^{-1}(D)=D$ for every prime divisor $D\in S(X)$.
\end{lemma}

\begin{proof}
$f^{-1}(S(X))=S(X)$ is from Step 4 of the proof of \cite[Lemma 2.9]{Zha10} which holds true for any surjective endomorphism.
Take any $D \in S(X)$.
Then $f^{-i}f^i(D)=D$ for any $i>0$.
Hence, $D$ is $f^{-1}$-periodic by
\cite[Lemma 3.5]{MZ20}. Thus $S(X)$ consists of
$f^{-1}$-periodic prime divisors and is hence a finite set by \cite[Proposition 3.6]{MZ20}.
\end{proof}

\begin{lemma}[{cf.~\cite[Propositions 1.1, 2.9 and Lemma 3.4]{MZ18}}]\label{lem-pol-mod>1}
Let $f$ be a surjective endomorphism of a projective variety $X$ and $q>0$.
Then the following are equivalent.
\begin{itemize}
\item[(1)] $f^*|_{\N^1(X)}$ is diagonalizable with all the eigenvalues being of modulus $q$.
\item[(2)] $f^*B \equiv qB$ for some big $\mathbb{R}$-Cartier divisor $B$.
\end{itemize}
If further $q>1$ is an integer, then the above are equivalent to
\begin{itemize}
\item[(3)] $f$ is $q$-polarized, i.e., $f^*H\sim qH$ for some ample Cartier divisor $H$ (and $q>1$).
\end{itemize}
In particular, $\deg (f) = q^{\dim X}$ if one of the above conditions holds.
\end{lemma}

\begin{lemma}[{cf.~\cite[Theorem 1.1]{Men20}, \cite[Propositions 3.4 and 3.7]{MZ19b}}]\label{lem-int-amp-equiv}
Let $f$ be a surjective endomorphism of a projective variety $X$.
Then the following are equivalent.

\begin{enumerate}
\item $f$ is int-amplified, i.e., $f^*L-L=H$ for some ample Cartier divisors $L$ and $H$.
\item All the eigenvalues of $\varphi:=f^*|_{\N^1(X)}$ are of modulus greater than $1$.
\item All the eigenvalues of $f^*|_{\Pic(X) \otimes_{\Z} \Q}$ are of modulus greater than $1$.
\item If $C$ is a $\varphi$-invariant convex cone in $\N^1(X)$, then $\emptyset\neq(\varphi-\textup{id}_{\N^1(X)})^{-1}(C)\subseteq C$.
\end{enumerate}
\end{lemma}

In our Fano setting, ``int-amplified'' is equivalent to ``amplified''.

\begin{lemma}\label{lem-MDS-amp}
Let $f$ be a surjective endomorphism of a normal projective variety $X$ with the nef cone $\textup{Nef}(X)$ being a rational polyhedron (this is the case when $X$ is Fano).
Then $f$ is amplified if and only if it is int-amplified.
\end{lemma}

\begin{proof}
One direction is clear by the definition.
Suppose $\dim(X)>0$ and $A=f^*L-L$ is ample for some Cartier divisor $L$.
Since $\textup{Nef}(X)$ is a rational polyhedron, $f^*$ fixes each extremal ray of $\Nef(X)$ after iteration.
Note that $\Nef(X)$ spans $\N^1(X)$.
Then $\N^1(X)$ admits a nef integral basis $\{D_1,\cdots,D_m\}$ such that $f^*D_i\equiv \lambda_iD_i$ for positive integers $\lambda_i$.

Write $L\equiv \sum_{i=1}^m a_iD_i$ and $L':=\sum_{i=1}^m D_i$.
Then $A=f^*L-L \equiv \sum a_i(\lambda_i-1)D_i$ and $A':=f^*L'-L' \equiv \sum (\lambda_i-1)D_i$.
For $c\gg 1$, our $cA' \equiv A+\sum (c-a_i)(\lambda_i-1)D_i$ and $cL' \equiv A+\sum (c-a_i(\lambda_i-1))D_i$ are both sums of the ample divisor $A$ and some nef divisors.
Hence, $A'$ and $L'$ are ample.
Also, by the construction of $A'$, we have $f^*(cL')-(cL')=cA'$.
So $f$ is int-amplified.
\end{proof}

We recall the well-known result (\cite[Theorem 3.5]{Mor82}) on the Fano contractions of smooth projective threefolds.
For our purpose, we focus on the rationally connected case.

\begin{lemma}\label{lem-base-rat}
Let $X$ be a smooth projective threefold and $X \to Y$ the Fano contraction of a $K_X$-negative extremal ray to a normal projective variety
$Y$.
Suppose $X$ is rationally connected
(this is the case when $X$ is Fano or of Fano type).
Then the following hold.
\begin{enumerate}
\item[(1)]
If $\dim (Y) = 1$, then $Y \cong \PP^1$.
\item[(2)]
Suppose that $\dim (Y) = 2$.
Then $Y$ is a smooth rational surface;
if further $Y$ has Picard number one, then
$Y \cong \PP^2$.
\end{enumerate}
\end{lemma}

\begin{proof}
Since $X$ is rationally connected, so is $Y$; hence $Y$ is rational when
$\dim (Y) \le 2$.
Then (1) follows.
For (2), just note that $Y$ is smooth by \cite[Theorem 3.5]{Mor82}.
\end{proof}

Given a smooth projective toric variety, its toric boundary is always nice.
\begin{lemma}\label{lem-SNC-tor-bd}
Let $(X, B)$ be a toric pair with $X$ being smooth projective.
Then $B$ is of simple normal crossing, i.e., the pair $(X, B)$ is log smooth,
and $K_X + B \sim 0$.
If $X$ is further a surface, then $B$ is a simple loop of smooth rational curves.
\end{lemma}

\begin{proof}
The first part is from \cite[p.~360 and Theorem 8.2.3]{CLS11} while the second follows from
the first and the adjunction (cf.~\cite[Proposition 3.2.7]{CLS11} and Notation \ref{notation2.1}).
\end{proof}

The proposition below should be well-known; we give the proof for readers' convenience.

\begin{proposition}\label{prop-toric-pair}
Let $X:=\mathbb{P}_Y(\mathcal{E})\xrightarrow{\tau} Y$ be a splitting $\mathbb{P}^1$-bundle over a normal projective variety $Y$.
Suppose that $(Y, \Delta_Y)$ is a toric pair.
Then $X$ is toric and $(X, S_0+S_1+\tau^{-1}(\Delta_Y))$ is a toric pair for any two disjoint (cross) sections $S_0, S_1$ of $\tau$.
\end{proposition}

\begin{proof}
Let $T(Y)=Y\backslash \Delta_Y$ and fix a toric action $T(Y)$ on $Y$.
Write $\mathcal{E}=\mathcal{L}_0\oplus \mathcal{L}_1$ as a direct sum of two line bundles on $Y$.
Denote by $\tau_i: \mathcal{L}_i\to Y$ the induced fibre bundles.
For each $x\in X$, we write $x=(y, [v_0:v_1])$ with $y=\tau_i(x)$ and $v_i\in (\mathcal{L}_i)_y:=\tau_i^{-1}(y)$.
Note that $\mathcal{L}_i\cong \mathcal{O}_Y(D_i)$ for some Cartier divisors $D_i$ with support contained in $\Delta_Y$ which are $T(Y)$-invariant.
Then $\mathcal{L}_i$ are $T(Y)$-equivariant line bundles, i.e., there are $T(Y)$-actions on $\mathcal{L}_i$ such that $\tau_i: \mathcal{L}_i\to Y$ are $T(Y)$-equivariant
and the $T(Y)$-actions on the fibres of $\tau_i$ are linear (cf.~\cite[\S 2]{Oda88} or \cite[\S 2.1]{KD19}).
Thus we can construct a faithful action of a torus $T:=T(Y)\times k^*=\{(g,t)\}$ on $X$ via:
$$(g, t)\cdot (y, [v_0:v_1])= (g(y), [g(v_0): t\cdot g(v_1)])$$
where the $T(Y)$ action on $X$ is determined by $\mathcal{L}_i$.
So $X$ is toric.

Let $S_0$ and $S_1$ be two disjoint sections of $\tau$.
They are uniquely determined by sub line bundles $\mathcal{H}_0, \mathcal{H}_1$ of $\mathcal{E}$.
Since $S_0\cap S_1=\emptyset$, we have $(\mathcal{H}_0)_y\neq (\mathcal{H}_1)_y$ for any $y\in Y$.
Therefore, $\mathcal{E}=\mathcal{H}_0\oplus \mathcal{H}_1$.
Updating the above $T$-action by using the new decomposition, we see that $\tau^{-1}(\Delta_Y), S_0, S_1$ are $T$-invariant and $T$ acts transitively on $X\backslash (S_0 \cup S_1 \cup \tau^{-1}(\Delta_Y))$.
\end{proof}

The following lemma allows us to focus on polarized endomorphisms whenever given non-isomorphic surface endomorphisms.

\begin{lemma}\label{lem-noniso-nonpol}
Let $S$ be a smooth projective rational surface admitting a non-isomorphic surjective endomorphism $f$ which is not polarized even after iteration.
Then $S\cong\mathbb{P}^1\times\mathbb{P}^1$.	
\end{lemma}

\begin{proof}
By \cite[Theorem 5.4]{MZ19b}, $\rho(S)=2$ and hence $S$ has an $f$-equivariant (after iterating $f$) ruling $\pi:S\to C\cong \mathbb{P}^1$.
Suppose the contrary that $S\not\cong\mathbb{P}^1\times\mathbb{P}^1$.
Then the ruling admits a negative section $C_0$ which is an $f^{-1}$-invariant curve after iterating $f$ (cf.~\cite[Lemma 9 and Proposition 11]{Nak02} or \cite[Lemma 4.3]{MZ19b}).
Write $f^*C_0=qC_0$ with $q\ge 1$.
By the projection formula, $q^2 = \deg (f) >1$ and hence $q>1$.
Then $f|_{C_0}$ and hence $f|_C$ are $q$-polarized.
Let $F$ be a fibre of $\pi$.
Then $f^*F\sim qF$.
Since $C_0+F$ is big and $f^*(C_0+F)\sim q(C_0+F)$, our $f$ is $q$-polarized (cf.~Lemma \ref{lem-pol-mod>1}),
a contradiction!
\end{proof}

In Theorem \ref{thm-bh} below,
we apply \cite[Theorem 1.4]{BH14} and \cite[Theorem 6.2]{MZ19b} to give some restrictions on totally invariant prime divisors.
Locally, for example when $X$ is smooth, the log canonical pair $(X,\Delta)$ implies that $\Delta$ has at most two components containing a common codimensional $2$ subvariety of $X$ (cf.~\cite[Lemma 2.29]{KM98}).
Globally, for example when $X=\mathbb{P}^n$, the effectivity of $-(K_X+\Delta)$ implies that  $\deg (\Delta)\le n+1$ and hence $\Delta$ has at most $n+1$ components.

\begin{theorem}\label{thm-bh}
Let $f$ be an int-amplified endomorphism of a normal projective variety $X$.
Let $\Delta$ be an $f^{-1}$-invariant reduced divisor such that $K_X + \Delta$ is $\Q$-Cartier.
Then
\begin{enumerate}
\item  $(X,\Delta)$ has at worst log canonical singularities, and
\item  $-(K_X+\Delta)$ is ($\Q$-linearly equivalent to) an effective divisor.
\end{enumerate}
\end{theorem}

\begin{proof}
(1) follows immediately from Lemma \ref{lem-int-amp-equiv} and \cite[Theorem 1.4]{BH14}.
To prove (2), we consider the log ramification divisor formula
$K_X+\Delta=f^*(K_X+\Delta)+R'$
where $R'$ is a $\Q$-Cartier effective divisor.
Let $D:=-(K_X+\Delta)$.
Then
$f^*D-D= R'$.
By the same proof of \cite[Theorem 6.2]{MZ19b}, $D$ is effective; see also \cite[Propositions 3.4, 3.7]{MZ19b} and \cite[Proposition 3.2]{Men20}.
\end{proof}

Below is a criterion for a $\PP^1$-bundle to be algebraic. 

\begin{lemma}[{cf.~\cite[\S 3. Proposition and Corollary 1]{Ele81}}]\label{lem-p1-bundle-projective}
Let $\tau:X\to Y$ be a (smooth) $\mathbb{P}^1$-bundle over a  smooth projective rational variety $Y$.
Then $\tau$ is an algebraic $\mathbb{P}^1$-bundle, i.e., $X=\mathbb{P}_Y(\mathcal{E})$ for some  locally free rank-two sheaf $\mathcal{E}$  over $Y$.
\end{lemma}


Theorem \ref{thm-amerik2} and Proposition \ref{prop-ame-split} below are important in showing the splitting-ness of algebraic $\PP^1$-bundles admitting non-isomorphic surjective endomorphisms.
For certain $X$ as in Theorem \ref{thm-conic-split-p1p1-blp2},
we can dispose the semistability assumption in
Theorem \ref{thm-amerik2}.

\begin{theorem}[{cf.~\cite[Theorem 2, Proposition 2.4]{Ame03}}]\label{thm-amerik2}
Let $X=\PP_Y(\mathcal{E})\xrightarrow{\tau} Y$ be an algebraic $\mathbb{P}^1$-bundle  over a  smooth projective rational  variety $Y$.
If $X$ admits an int-amplified endomorphism $f$ and $\mathcal{E}$ is $H$-semistable for some ample divisor $H$, then $\tau$ is a splitting $\mathbb{P}^1$-bundle.
\end{theorem}

\begin{proof}
Since $f$ is int-amplified, we see that $\deg (f)>\deg (f|_Y)$ (cf.~Lemma \ref{lem-int-amp-equiv}).
Note that rational varieties are simply connected.
Hence, the $H$-semistable sheaf $\mathcal{E}$ splits by \cite[Theorem 2 and Proposition 2.4]{Ame03}.
\end{proof}

\begin{proposition}[{\cite[Proposition 3]{Ame03}}]\label{prop-ame-split}
Suppose $B=\mathbb{P}^n$ and $X=\mathbb{P}_B(\mathcal{E})$ with $\mathcal{E}$ a vector bundle of rank two.
Then there exists an endomorphism of $X$ of degree bigger than one if and only if $\mathcal{E}$ is a direct sum of two line bundles.
\end{proposition}

In what follows, we recall the known results of Conjectures \ref{main-conj-pn} and \ref{main-conj-linear}.

\begin{theorem}[{cf.~\cite{ARV99}, \cite{HM03}}]\label{thm-fano-rho=1}
Let $X$ be a smooth Fano threefold of Picard number one which admits a non-isomorphic surjective endomorphism.
Then $X\cong\mathbb{P}^3$.
\end{theorem}

\begin{theorem}\label{thm-linear-subspace}
Conjecture \ref{main-conj-linear} holds if one of the following cases occurs.
\begin{enumerate}
\item $X = \mathbb{P}^2$ (cf.~\cite{Gur03}).
\item $X = \mathbb{P}^3$; $Y$ is an $f^{-1}$-invariant prime divisor (cf.~\cite{Hor17}, \cite[Thm 1.5 (5)]{NZ10}).
\item $Y$ is a smooth hypersurface of $X$ (cf.~\cite{CL00}, \cite{Bea01} and \cite{Gur03}).
\end{enumerate}
\end{theorem}
At the end of this section, we recall the following well-known fact which will be heavily used in Section \ref{section-4-conic}.
\begin{proposition}[{cf.~\cite[Theorem 3.7 (4)]{KM98}}]\label{pro-cone-reduced-reducible}
Let $\pi:X\to Y$ be a Fano contraction of a $K_{X}$-negative extremal ray $R$ between projective varieties.
Suppose that $X$ is smooth.
Then for every prime divisor $Q$ on $Y$, its pullback $\pi^*Q$ is reduced and irreducible.
\end{proposition}
\begin{proof}
Suppose the contrary that $\pi^*Q=P_1+P_2$ where $P_1$ is a prime divisor and $P_2$ is a non-zero integral divisor on $X$.
Then for a general curve $\ell$ contracted by $\pi$ (and hence a generator of $R$) we have $P_i\cdot\ell=0$ for each $i$.
Since $X$ is smooth, both $P_i$ are effective Cartier divisors.
By the cone theorem (cf.~\cite[Theorem 3.7 (4)]{KM98}), there exist non-zero effective Cartier divisors $Q_i$ on $Y$ such that $P_i=\pi^*Q_i$.
Therefore, $Q=Q_1+Q_2$ being non-reduced or reducible gives the contradiction.
\end{proof}

\section{Totally periodic prime divisors}\label{section-3-tid}
This section describes totally periodic
divisors in surfaces and splitting $\mathbb{P}^1$-bundles.

\begin{lemma}\label{lem-p1p1-ticurve}
Let $f$ be a non-isomorphic surjective endomorphism of $X=Y\times Z\cong\mathbb{P}^1\times \mathbb{P}^1$.
Let $C$ be an $f^{-1}$-periodic curve in $X$.
Then $C$ is either $\{\cdot\}\times \mathbb{P}^1$ or $\mathbb{P}^1\times\{\cdot\}$.
\end{lemma}

\begin{proof}
Denote by $p_Y:X\to Y$ and $p_Z:X\to Z$ the natural projections.
After iteration, we may assume $C$ is $f^{-1}$-invariant and $f$ splits into the form $g\times h$ with $g$ and $h$ being surjective endomorphisms.

Suppose the contrary that $C$ dominates both $Y$ and $Z$, via $p_Y$ and $p_Z$.
Then $q: = \deg (g)=\deg (f|_C)=\deg (h)$.
Since $f$ is non-isomorphic, we have $q > 1$.
Hence,
$f$ is $q$-polarized.
So, $f^*C=qC$ (cf.~Lemma \ref{lem-pol-mod>1}) and thus $C$ is a component of the ramification divisor $R_f$ of $f$.
Since $f = g \times h$, we have $R_f=p_Y^*R_g+p_Z^*R_{h}$ where  $R_g$ and $R_h$ are the ramification divisors of $g$ and $h$, respectively.
But then, $C$ is of the form $\{\cdot\}\times \mathbb{P}^1$ or $\mathbb{P}^1\times\{\cdot\}$,  a contradiction to our assumption.
This proves the lemma.
\end{proof}

Now we consider general smooth projective rational surfaces.

\begin{theorem}\label{thm-surface-toricpair}
Let $X$ be a smooth projective rational surface admitting an int-amplified endomorphism $f$.
Then $X$ is toric and there is a toric pair $(X,\Delta)$ such that $\Delta$ contains the union $\Sigma$ of all the $f^{-1}$-periodic prime divisors.

In particular, if $\rho(X) = 2$, then $\Delta$ is a union of two disjoint (cross) sections of a fixed ruling of $X$ and two distinct fibres.
\end{theorem}

\begin{proof}
Since there are only finitely many $f^{-1}$-periodic subvarieties by \cite[Corollary 3.8]{MZ20}, we may assume that they are all $f^{-1}$-invariant, after iterating $f$. Further,
$(X, \Sigma)$ is lc and $-(K_X+\Sigma)$ is $\Q$-linearly equivalent to an effective divisor by Theorem \ref{thm-bh}.

\par \vskip 0.3pc \noindent
\textbf{Case $\rho(X)=1$.}
Then $X\cong \mathbb{P}^2$ and $\Sigma$ is a union of at most three lines by Theorem \ref{thm-linear-subspace}; noting  that  $\deg (\Sigma)\le 3$.
Since $(X, \Sigma)$ is lc, $\Sigma$ is contained in some $\Delta$, a union of three lines with no common intersection (looking like \mstriangle{}).
Clearly, $(X, \Delta)$ is a toric pair.

\par \vskip 0.3pc \noindent
\textbf{Case $\rho(X)=2$.}
Suppose $X\cong \mathbb{P}^1\times \mathbb{P}^1$.
By Lemma \ref{lem-p1p1-ticurve},
$\Sigma$ is contained in some $\Delta$, a union of two fibres each from two projections of $X$ (looking like \mssharp{}).
Clearly, $(X, \Delta)$ is a toric pair.

Suppose $X\not\cong \mathbb{P}^1\times \mathbb{P}^1$.
Consider the $f$-equivariant (after iteration) ruling $\tau:X\to Z\cong\mathbb{P}^1$ (which is also a splitting $\mathbb{P}^1$-bundle).
Write $\Sigma=\Sigma_h\cup\Sigma_v$ where $\Sigma_h$ (resp.~$\Sigma_v$) is the union of components of $\Sigma$ dominating (resp.~not dominating) $Z$.
Since $f|_Z$ is also int-amplified on $Z \cong \PP^1$ (cf.~Lemma \ref{lem-int-amp-equiv}),
there are at most two $(f|_Z)^{-1}$-periodic points.
So $\Sigma_v$ is a union of at most two fibres of $\tau$ (cf.~\cite[Lemma 7.5]{CMZ20}).
In particular, there exists a toric pair $(Z, \Delta_Z)$ such that $\Sigma_v\subseteq \tau^{-1}(\Delta_Z)$.

Let $F\cong \mathbb{P}^1$ be a fibre of $\tau$ which is nef on $X$.
Then  $-(K_X+\Sigma)$ being $\Q$-linearly equivalent to an effective divisor (cf.~Theorem \ref{thm-bh}) implies that
$$\Sigma_h\cdot F=\Sigma\cdot F\le -K_X\cdot F=-\deg (K_F)=2$$
by the adjunction formula.
Note that $\Sigma_h$ contains the unique negative section $C_0$ of $\tau$.
Therefore, either $\Sigma_h=C_0$ or $\Sigma_h$ is a union of $C_0$ and another section $C$ of $\tau$.
In the first case, we are done by applying Proposition \ref{prop-toric-pair} for $\Delta=C_1\cup C_0\cup \tau^{-1}(\Delta_Z)$ where $C_1$ is a section of $\tau$ disjoint with $C_0$ (cf.~\cite[Chapter V, Theorem 2.17]{Har77}).

Similarly, in the second case, it suffices for us to show that $C\cap C_0=\emptyset$.
Suppose the contrary and we take $x\in C\cap C_0$ which is $f^{-1}$-invariant after iteration.
Note that $\tau^{-1}(\tau(x))$ is $f^{-1}$-invariant (cf.~\cite[Lemma 7.5]{CMZ20}).
Then $\Sigma$ contains three curves $C, C_0, \tau^{-1}(\tau(x))$ with a common intersection point $x$.
In particular, $(X, \Sigma)$ is not lc at $x$ (cf.~\cite[Lemma 2.29]{KM98}), a contradiction to Theorem \ref{thm-bh}.
So $C\cap C_0=\emptyset$ and the second case is done by applying Proposition \ref{prop-toric-pair} again for $\Delta=C\cup C_0\cup \tau^{-1}(\Delta_Z)$.

\par \vskip 0.3pc \noindent
\textbf{Case $\rho(X)\ge 3$.}
By Lemma \ref{lem-equiv-MMP}, running an $f$-equivariant (after iterating $f$) MMP, we have a composition $\pi:X\to Y$ of smooth blowdowns of ($f^{-1}$-invariant successively) $(-1)$-curves with $Y$ being a Hirzebruch surface and the ruling $\tau_Y:Y\to Z\cong \mathbb{P}^1$.
Set $\tau:=\tau_Y\circ \pi$,
and $f_Y:= f|_Y$.
Since $\rho(X)\ge 3$, there is an $f_Y^{-1}$-invariant point $y_0$ in $Y$.

We claim that $\tau_Y: Y \to Z$ admits at least one $f_Y^{-1}$-invariant (cross) section $C_Y$.
Indeed, if $Y\cong \mathbb{P}^1\times\mathbb{P}^1$, then the horizontal (fibre) section $C_Y$ of $\tau_Y$ containing $y_0$ is $f_Y^{-1}$-invariant (cf.~Lemma \ref{lem-p1p1-ticurve}).
If $Y\not\cong \mathbb{P}^1\times\mathbb{P}^1$, then the unique negative section $C_Y$ of $\tau_Y$ is $f_Y^{-1}$-invariant.
So the claim is proved.
Let $C_X \subseteq X$ be the strict transform of $C_Y$.
Then $C_X \subseteq \Sigma$.

Since $\rho(X)\ge 3$, $\tau$ admits at least one reducible fibre $F=\bigcup_{i=1}^n F_i$ with $F_i\cong \mathbb{P}^1$ for all $i$ and $n\ge 2$.
By the generic smoothness, $\tau$ has only finitely many reducible fibres.
So $F$ is $f^{-1}$-periodic by \cite[Lemma 7.4]{CMZ20} and hence $C_X\cup F\subseteq \Sigma$.
After iteration, we may assume $f^{-1}(F_i)=F_i$ for each $i$.

We claim that $F\cup C_X$ is a linear chain of smooth rational curves with dual graph
\begin{center}
\begin{tikzpicture}[x=1.3cm,y=1cm]
\draw (0,0) node{{\tiny $\circ$}};
\draw (1,0) node{{\tiny $\circ$}};
\draw (2,0) node{{\tiny $\circ$}};
\draw (3,0) node{{\tiny $\circ$}};
\draw (3.5,0) node{{\tiny $\cdots$}};
\draw (4,0) node{{\tiny $\circ$}};
\draw (5,0) node{{\tiny $\circ$}};

\draw (0+0.05,0) -- (1-0.05,0);
\draw (1+0.05,0) -- (2-0.05,0);
\draw (2+0.05,0) -- (3-0.05,0);
\draw (4+0.05,0) -- (5-0.05,0);

\draw (0,-0.5) node{{\tiny $C_X$}};
\draw (1,-0.5) node{{\tiny $F_1$}};
\draw (2,-0.5) node{{\tiny $F_2$}};
\draw (3,-0.5) node{{\tiny $F_3$}};
\draw (4,-0.5) node{{\tiny $F_{n-1}$}};
\draw (5,-0.5) node{{\tiny $F_n$}};
\end{tikzpicture}
\end{center}
Denote by $F_0:=C_X$ for convenience.
Note that $f|_{F_i}$ (like $f$) is int-amplified for each $0\le i\le n$.
So each $F_i$ admits at most two $(f|_{F_i})^{-1}$-periodic points.
Since $F_i\cap F_j$ is $(f|_{F_i})^{-1}$-invariant for any $j$, we have $\# F_i \cap (F_0\cup F\backslash F_i)\le 2$ for each $0\le i\le n$.
Recall that $\pi(F)$ intersects with $C_Y=\pi(C_X)=\pi(F_0)$ transversally at only one point.
So $C_X$ intersects $F$ transversally at only one point and $C_X \cup F$ is a divisor of simple normal crossing.
Finally, note that $C_X \cup F$ is connected.
So the claim follows.

We claim further that one can find some $F_{i_0}^2=-1$ with $0<i_0<n$.
We show by induction on $n$.
If $n=2$, then $F_1$ and $F_2$ are both $(-1)$-curves, since $\pi(F)^2=0$.
In general, note that at least one component of $F$ is a $(-1)$-curve because our $\pi$ is a composition of blowups.
Suppose $F_n$ is the only $(-1)$-curve contained in $F$.
Let $\widetilde{\pi}:X\to \widetilde{X}$ be the ($f$-equivariant) contraction of $F_n$.
Then we have the following dual graph of $\widetilde{\pi}(C_X \cup F)$:
\begin{center}
\begin{tikzpicture}[x=1.3cm,y=1cm]
\draw (0,0) node{{\tiny $\circ$}};
\draw (1,0) node{{\tiny $\circ$}};
\draw (2,0) node{{\tiny $\circ$}};
\draw (3,0) node{{\tiny $\circ$}};
\draw (3.5,0) node{{\tiny $\cdots$}};
\draw (4,0) node{{\tiny $\circ$}};
\draw (5,0) node{{\tiny $\circ$}};

\draw (0+0.05,0) -- (1-0.05,0);
\draw (1+0.05,0) -- (2-0.05,0);
\draw (2+0.05,0) -- (3-0.05,0);
\draw (4+0.05,0) -- (5-0.05,0);

\draw (0,-0.5) node{{\tiny $\widetilde{\pi}(C_X)$}};
\draw (1,-0.5) node{{\tiny $\widetilde{\pi}(F_1)$}};
\draw (2,-0.5) node{{\tiny $\widetilde{\pi}(F_2)$}};
\draw (3,-0.5) node{{\tiny $\widetilde{\pi}(F_3)$}};
\draw (4,-0.5) node{{\tiny $\widetilde{\pi}(F_{n-2})$}};
\draw (5,-0.5) node{{\tiny $\widetilde{\pi}(F_{n-1})$}};
\end{tikzpicture}
\end{center}
Note that $\widetilde{\pi}(F_i)^2=F_i^2\neq -1$ for $0<i<n-1$, a contradiction by induction on $n$.
So the claim is proved.

By the above claim, let $\pi': X \to X'$ be the ($f$-equivariant) contraction of the $(-1)$-curve $F_{i_0}$ over $Z$ (which may no longer be over $Y$).
Then $X'$ is still smooth and $\rho(X')=\rho(X)-1$.
Let $\Sigma'$ be the union of
$(f|_{X'})^{-1}$-periodic curves in $X'$.
Then $\Sigma'=\pi'(\Sigma)$ (cf.~\cite[Lemma 7.5]{CMZ20}).
By induction on the Picard number, there is a toric pair $(X', \Delta')$ such that $\Sigma'\subseteq \Delta'$.
Denote by $F_0:=C_X$.
Since $0<i_0<n$, our $\pi'(F_{i_0})=\pi'(F_{i_0-1})\cap \pi'(F_{i_0+1})$ is an intersection of two components of $\Sigma'$ (and also $\Delta'$).
In particular, $\pi'$ is a toric blowup and $(X, \Delta)$ is a toric
pair with $\Delta$ being the union of $F_{i_0}$ and the strict transform of $\Delta'$.
Clearly, $\Sigma\subseteq \Delta$.
This proves the theorem.
\end{proof}

\begin{theorem}\label{thm-boundary-p1bundle}
Let $\tau:X\to Y$ be a splitting $\mathbb{P}^1$-bundle over a smooth projective rational surface
$Y$.
Suppose that $X$ admits an int-amplified endomorphism $f$.
Then $X$ is toric and there is a toric pair $(X,\Delta)$ such that $\Delta$ contains the union $\Sigma$ of all the $f^{-1}$-periodic prime divisors.
\end{theorem}

\begin{proof}
Since there are only finitely many $f^{-1}$-periodic subvarieties by \cite[Corollary 3.8]{MZ20},
we may assume that components of $\Sigma$ are all $f^{-1}$-invariant, after iterating $f$.
Write $\Sigma=\Sigma_h\cup\Sigma_v$ where $\Sigma_h$ (resp.~$\Sigma_v$) is the union of components of $\Sigma$ dominating (resp.~not dominating) $Y$.
Note that $\tau$ is a Fano contraction of some $K_X$-negative extremal ray.
After iteration, $\tau$ is $f$-equivariant and $f|_Y$ (like $f$) is also int-amplified (cf.~Lemmas \ref{lem-equiv-MMP} and \ref{lem-int-amp-equiv}).
Note that $\Sigma_v=\tau^{-1}(\Sigma_Y)$ where $\Sigma_Y$ is the union of $(f|_Y)^{-1}$-periodic curves in $Y$ (cf.~\cite[Lemma 7.4]{CMZ20}).
By Theorem \ref{thm-surface-toricpair}, there exists a toric pair $(Y, \Delta_Y)$ such that $\Sigma_v\subseteq \tau^{-1}(\Delta_Y)$.
By Proposition \ref{prop-toric-pair}, it suffices for us to show that $\Sigma_h$ is contained in two disjoint sections of $\tau$.

Assume first $\tau$ is a trivial bundle
so that $X = Y \times Z \cong Y \times \PP^1$
and we may assume that $f = g \times h$, after iteration (cf.~Lemma \ref{lem-equiv-MMP}).
By Lemma \ref{lem-int-amp-equiv} and the same proof of Lemma \ref{lem-p1p1-ticurve},
$\Sigma_h \subseteq \Supp R_f=(\Supp R_g\times Z)\cup (Y\times \Supp R_h)$.
Hence, $\Sigma_h \subseteq Y\times \Supp R_h$.
Since $h$ on $Z \cong \PP^1$ (like $f$) is also int-amplified,
$\Supp R_h$ has at most two $h^{-1}$-periodic points.
So $\Sigma_h$ is contained in two disjoint fibres of $X \to Z$ which are two disjoint sections of $\tau$ as desired.

Assume next $\tau$ is non-trivial.
Write $X=\mathbb{P}_Y(\mathcal{E})$ with $\mathcal{E}=\mathcal{O}_Y\oplus \mathcal{L}$ such that $\mathcal{L}$ is not pseudo-effective.
Let $S_0$ be the (cross) section of $\tau$ determined by the projection $\mathcal{O}_Y\oplus\mathcal{L}\to\mathcal{L}$.
Note that $\mathcal{O}_X(S_0)|_{S_0}\cong (\tau|_{S_0})^*\mathcal{L}$ is not pseudo-effective.
By Lemma \ref{lem-non-pe-ti}, $S_0$ is $f^{-1}$-periodic and hence $S_0\subseteq \Sigma_h$.

Let $X_y:=\tau^{-1}(y)\cong \mathbb{P}^1$ be a fibre of $\tau$.
By Theorem \ref{thm-bh}, $-(K_X+\Sigma)$ is $\Q$-linearly equivalent to an effective divisor.
Since $X_y$ is movable,
we have $-(K_X+\Sigma)\cdot X_y\ge 0$ and hence
$$\Sigma_h\cdot X_y=\Sigma\cdot X_y\le -K_X\cdot X_y=-\deg (K_{X_y})=2$$
by the adjunction formula.
Note that $\Sigma_h$ contains $S_0$.
Therefore, either $\Sigma_h=S_0$ or $\Sigma_h$ is a union of $S_0$ and another prime divisor $S$ with $S\cdot X_y=1$ (which is not known to be a section of $\tau$ currently).
In the first case, we are done by applying Proposition \ref{prop-toric-pair} for $\Delta=S_1\cup S_0\cup \tau^{-1}(\Delta_Y)$ where $S_1$ is a section of $\tau$ disjoint with $S_0$ by the splitting structure of $\mathcal{E}$ (cf.~the proof of Proposition \ref{prop-toric-pair}).
Similarly, in the second case, it suffices for us to show that $S$ is a section of $\tau$ and $S\cap S_0=\emptyset$.

We first claim that $S\cap S_0=\emptyset$ if $\Sigma_h$ contains another prime divisor $S\neq S_0$.
Suppose the contrary and let $C$ be an irreducible curve in  $S\cap S_0$ which is $f^{-1}$-invariant after iteration.
Then $F:=\tau^{-1}(\tau(C))$ is an $f^{-1}$-invariant divisor (cf.~\cite[Lemma 7.5]{CMZ20}).
Note that $C\subseteq F\cap S\cap S_0$.
Then the pair $(X,S+S_0+F)$ is not lc at $C$ (cf.~\cite[Lemma 2.29]{KM98}).
However, this contradicts Theorem \ref{thm-bh}.
The claim is proved.

The above claim further implies that $S$ (if exists) does not contain any fibre of $\tau$.
Then $S\cdot X_y=1$ implies that $S\cap X_y$ is a single point for any $y\in Y$.
Therefore, $S$ is a section disjoint with $S_0$ as desired.
This proves the theorem.
\end{proof}

\begin{corollary}\label{cor-ti-rational}
Let $f$ be a non-isomorphic surjective endomorphism of a smooth projective rational surface $X$.
Let $C$ be an $f^{-1}$-periodic curve.
Then $C$ is a smooth rational curve.
In particular, if $X$ is further assumed to be Fano, then $C^2\le 1$.
\end{corollary}

\begin{proof}
If $f$ is not int-amplified, then $X\cong\mathbb{P}^1\times\mathbb{P}^1$ by Lemma \ref{lem-noniso-nonpol}, and our corollary follows from Lemma \ref{lem-p1p1-ticurve}.
If $f$ is int-amplified, then there exists a toric pair $(X,\Delta)$ such that $\Delta$ contains $C$ (cf.~Theorem \ref{thm-surface-toricpair});
thus $C \cong \PP^1$ by Lemma \ref{lem-SNC-tor-bd}.

Suppose further that $X$ is Fano.
By Lemma \ref{lem-equiv-MMP},
there is an $f$-equivariant (after iteration) composition $\pi:X\to Y$ of smooth blowdowns such that $Y\cong \mathbb{P}^2$ or $\mathbb{P}^1\times \mathbb{P}^1$ since $X$ and hence $Y$ are still Fano.
Note also that either $C$ is an exceptional curve of $\pi$ or $\pi(C)$ is an $(f|_Y)^{-1}$-periodic curve
as described in Theorem \ref{thm-linear-subspace} or Lemma \ref{lem-p1p1-ticurve}  (cf.~\cite[Lemma 7.5]{CMZ20}).
Thus $C^2 = -1$ or $C^2 \le \pi(C)^2 \le 1$.
\end{proof}

\section{Fano contractions to surfaces}\label{section-4-conic}

The whole section is devoted to the proof of the following theorem.

\begin{theorem}\label{thm-conic}
Let $\tau:X\to Y$ be a Fano contraction of an extremal ray of a smooth Fano threefold $X$ with $\dim(Y)=2$.
Suppose $X$ admits a non-isomorphic surjective endomorphism $f$.
Then $X$ is an algebraic $\mathbb{P}^1$-bundle $\mathbb{P}_Y(\mathcal{E})$ over $Y$.
\end{theorem}

\par \noindent
\textbf{Now we begin to prove Theorem \ref{thm-conic}.} Recall that $Y$ is a
smooth rational surface and
$\tau$ is a conic bundle with the discriminant locus denoted as $\Delta_{\tau}$
(cf.~\cite[Theorem 3.5]{Mor82} and Lemma \ref{lem-base-rat}).
Besides, we can descend $f: X \to X$ to $g:= f|_Y$ after iteration by Lemma \ref{lem-equiv-MMP}.
By Lemma \ref{lem-p1-bundle-projective}, to prove Theorem
\ref{thm-conic},
it suffices for us to show $\Delta_{\tau}=\emptyset$.

\begin{notation}\label{notation-conic}
We will use the following notation throughout this section.
\begin{itemize}

\item[(1)]
$X$ is a smooth Fano threefold, and $Y$ is a smooth projective rational surface;
the conic bundle $\tau : X \to Y$ is flat
(cf.~\cite[Proposition 1.2]{Bea77}); $\rho(X) = \rho(Y)+1$.

\item[(2)]
$f:X\to X$ is a non-isomorphic surjective endomorphism and $g:=f|_Y$.

\item[(3)]
$\Delta_\tau$ is the discriminant locus of $\tau$, which is a union of irreducible curves $C_i$.

\item[(4)]$
S_i:=\tau^{-1}C_i=\tau^*C_i$ is
reduced and irreducible (cf.~\cite[Theorem 3.7]{KM98}).

\item[(5)]The fibres $X_y$ with $y\in \Delta_\tau$ satisfy the following: If $y$ is a smooth point of $\Delta_\tau$, then $X_y$ is a reducible conic
in $\PP^2$.
If $y$ is a double point of $\Delta_\tau$, then $X_y$ is isomorphic to a double line in $\PP^2$ (cf.~\cite[Proposition 1.2]{Bea77}).

\item[(6)] If $f$ is polarized (or int-amplified), then so is $g: Y \to Y$ (cf.~Lemmas \ref{lem-pol-mod>1} and \ref{lem-int-amp-equiv}).
\end{itemize}	
\end{notation}

\begin{proposition}[{cf.~\cite[Claim 3.3]{Zha12}}]\label{prop-zhang}
Suppose that $\Delta_\tau\neq\emptyset$.
Then $\deg (g)>1$.
\end{proposition}

\begin{lemma}\label{lem-simple-loop}
Suppose that $\Delta_\tau\neq\emptyset$.
Then $\Delta_{\tau}$ is $g^{-1}$-invariant and
a union of (at least two) $g^{-1}$-periodic smooth rational curves $C_i$ with
$C_i \cdot (\Delta_\tau - C_i) \ge 2$.
\end{lemma}

\begin{proof}
By Proposition \ref{prop-zhang}, $\deg (g)>1$.
By \cite[Lemma 7.4]{CMZ20}, $g^{-1}(\Delta_\tau)=\Delta_{\tau}$. Thus all $C_i$ are $g^{-1}$-periodic.
Hence, $C_i$ are all smooth rational curves by Corollary \ref{cor-ti-rational}.
Note that each $C_i$ meets at least two points of other irreducible components of $\Delta_\tau$ (cf.~\cite[Remark 4.2]{Miy83}).
So the lemma follows.
\end{proof}

\begin{lemma}\label{lem-Y-toric}
Suppose that $\Delta_\tau\neq\emptyset$ and $g$ is int-amplified.
Then $\Supp R_g=\Delta_\tau$ and $(Y,\Delta_\tau)$ is a toric pair.
In particular, $\Delta_\tau$ is a simple loop of smooth rational curves, and $K_Y + \Delta_{\tau} \sim 0$ (cf.~Lemma \ref{lem-SNC-tor-bd}).
\end{lemma}

\begin{proof}
By Lemma \ref{lem-simple-loop} and Theorem \ref{thm-surface-toricpair}, there is a toric pair $(Y, \Delta)$ such that $\Delta_{\tau}\subseteq \Delta$.
Note that $\Delta$ is a simple loop of smooth rational curves, i.e, each component of $\Delta$ intersects (transversally) exactly two other components of $\Delta$ at two distinct points (cf.~Lemma \ref{lem-SNC-tor-bd}).
So Lemma \ref{lem-simple-loop} further implies that $\Delta_{\tau}=\Delta$, which is $g^{-1}$-invariant.
Since $K_Y+\Delta_{\tau}\sim 0$,
we obtain the log ramification divisor formula
$K_Y+\Delta_{\tau}=g^*(K_Y+\Delta_{\tau})$.
Therefore, $\Supp R_g=\Delta_\tau$.
So the lemma is proved.
\end{proof}

\begin{lemma}\label{lem-p2orp1p1}
Suppose that $\Delta_\tau\neq\emptyset$.
Then
$Y\cong\mathbb{P}^2$ or $Y\cong\mathbb{P}^1\times\mathbb{P}^1$.
\end{lemma}

\begin{proof}
Suppose the contrary that the lemma is false.
After iteration, $g$ is $q$-polarized by Lemma \ref{lem-noniso-nonpol} and $Y$ admits a $g^{-1}$-invariant negative curve $E$ (cf.~\cite{Nak02} or \cite[Lemma 4.3]{MZ19b}).
Then $g^*E=qE$ and $E \subseteq \Supp R_g$.
By Lemma \ref{lem-Y-toric}, $E$ is a smooth rational curve contained in $\Delta_{\tau}$.
Hence,  $E$ is a rational connected component of $\Delta_\tau$ by \cite[Corollary 6.7]{MM83}.
This contradicts Lemma \ref{lem-simple-loop}.
So the lemma holds true.
\end{proof}

\begin{proposition}\label{prop-no-div}
Suppose that $\Delta_\tau\neq\emptyset$ and $f$ is int-amplified.
Then $X$ has no divisorial contraction.
\end{proposition}

We now prove Proposition \ref{prop-no-div}.
It will last till Claim \ref{cl-neq0}.
Suppose the contrary that $X$ has a divisorial contraction $\pi$ with $E$ the exceptional divisor.
After iteration, we may assume $f^{-1}(E)=E$
(cf.~Lemma \ref{lem-equiv-MMP}).
By \cite[Theorem 3.3]{Mor82}, $E$ is normal and $\mathbb{Q}$-factorial.
Note that $g$ and $f|_E$ (like $f$) are also int-amplified.
Let $T:=Y\backslash\Delta_{\tau}$ and $T_E:=E\backslash \tau^{-1}(\Delta_{\tau})$.
By Lemma \ref{lem-simple-loop}, $f^{-1}(T_E)=T_E$ after iteration.

\begin{claim}\label{cl-2-cover}
$\tau|_E:E\to Y$ is a generically finite surjective morphism of degree $2$.
\end{claim}

\begin{proof}
Consider the case $\tau(E)\neq Y$.
We have $E=\tau^*\tau(E)$ by the cone theorem (cf.~\cite[Theorem 3.7]{KM98}).
By Lemma \ref{lem-p2orp1p1}, $Y$ has no negative curves.
Then $\tau(E)$ and hence $E$ are nef, a contradiction to $-E$ being $\pi$-ample.
Thus $\tau(E)=Y$. Note that $\tau$ is a contraction of an extremal ray.
So $E$ is $\tau$-ample.

We still need to show that $E\cdot \ell=2$ for a general fibre $\ell$ of $\tau$.
Note that $-(K_X+E)$ is $\Q$-linearly equivalent to an effective divisor by Theorem \ref{thm-bh} and $\ell$ is movable.
So $-(K_X+E)\cdot \ell\ge 0$ and hence
$$E\cdot \ell\le -K_X\cdot \ell=2$$
by the adjunction formula.
On the other hand, since $\tau$ is a (flat) conic bundle, $\ell\equiv \ell'+\ell''$ where $\ell' \cup \ell''$ is a (reduced) reducible fibre lying over a smooth point of $\Delta_\tau$.
So we have
$$E \cdot \ell = E \cdot (\ell' + \ell'') \ge 1 + 1$$
where the second inequality is because $E$ is integral and $\tau$-ample.
Hence, $E \cdot \ell = 2$ and the claim is proved.
\end{proof}

\begin{claim}\label{cl-etale}
The restriction $f|_{T_E}$ is an \'etale morphism.
\end{claim}

\begin{proof}
Consider the following commutative diagram:
\[
\xymatrix{
T_E\ar@/^1pc/[rr]^{f|_{T_E}}\ar[rd]_{\tau|_{T_E}}\ar[r]_{\phi} &T'\ar[r]_{p_1}\ar[d]^{p_2} &T_E\ar[d]^{\tau|_{T_E}}\\
&T\ar[r]^{g|_T} &T
}
\]
where $T'$ is the fibre product of $g|_T$ and $\tau|_{T_E}$.
Since $g|_T$ is finite \'etale by Lemma \ref{lem-Y-toric},
so is $p_1$, hence $T'$ is normal.
Note that $\tau|_{T_E}$ is generically finite.
Then $\deg (f|_{T_E})=\deg (g|_T)=\deg (p_1)$, and hence $\deg (\phi)=1$, i.e., $\phi$ is birational.
Since $\phi$ (like $f|_{T_E}$) is also finite, and both $T_E$ and $T'$ are normal,
$\phi$ is an isomorphism.
So the claim is proved.
\end{proof}

\begin{claim}\label{cl-eq0}
$\tau^{-1}(\Delta_{\tau})|_E$ equals
$\tau^{-1}(\Delta_{\tau})\cap E$, a reduced divisor,
and
$K_E+\tau^{-1}(\Delta_{\tau})|_E\sim_{\mathbb{Q}}0$.
\end{claim}

\begin{proof}
By Claim \ref{cl-etale}, the ramification divisor of $f|_E$ has support contained in the $(f|_E)^{-1}$-invariant divisor $\tau^{-1}(\Delta_{\tau})\cap E$.
Thus, by the log ramification divisor formula,
\begin{equation}\label{cl4-10-eq1} \tag{\dag}
K_E+\tau^{-1}(\Delta_{\tau})\cap E=(f|_E)^*(K_E+\tau^{-1}(\Delta_{\tau})\cap E).
\end{equation}
This and Lemma \ref{lem-int-amp-equiv} imply $K_E+\tau^{-1}(\Delta_{\tau})\cap E \equiv0$,
since $f|_E$ is int-amplified.

On the other hand, by the log ramification divisor formula on $X$, we have
\begin{equation}\label{cl4-10-eq-adj}\tag{\dag\dag}
K_X+\tau^{-1}(\Delta_{\tau})+E=f^*(K_X+\tau^{-1}(\Delta_{\tau})+E)+R_f',
\end{equation}
where $R_f'$ is an effective divisor with no common component with $\tau^{-1}(\Delta_{\tau})+E$.
Restricting the above equation
to $E$ for both sides, we have
$$(K_X+\tau^{-1}(\Delta_{\tau})+E)|_E=(f|_E)^*(K_X+\tau^{-1}(\Delta_{\tau})+E)|_E+R_f'|_E,$$
where $R_f'|_E$ is effective since $E\not\subseteq \Supp R_f'$.
By the adjunction formula, we have
\begin{equation*}\label{cl4.10-log-onx}
K_E+\tau^{-1}(\Delta_{\tau})|_E=(f|_E)^*(K_E+\tau^{-1}(\Delta_{\tau})|_E)+R_f'|_E	,
\end{equation*}
which implies that
$-(K_E+\tau^{-1}(\Delta_{\tau})|_E)$ is pseudo-effective by Lemma \ref{lem-int-amp-equiv}.
Note that
\begin{equation*}\label{cl4-10-eq2}
-(K_E+\tau^{-1}(\Delta_{\tau})|_E)\le -(K_E+\tau^{-1}(\Delta_{\tau})\cap E)\equiv 0.
\end{equation*}
So the inequality above is in fact an equality. This proves the first half  of the
claim, while the second half follows from the first half,
Equation (\ref{cl4-10-eq1}) above
and Lemma \ref{lem-int-amp-equiv} (noting that $E$ is normal and $\mathbb{Q}$-factorial).
\end{proof}

\begin{claim}\label{cl-neq0}
$K_X+\tau^{-1}(\Delta_{\tau})+E\not\equiv 0$.
\end{claim}

\begin{proof}
Suppose the contrary that $K_X+\tau^{-1}(\Delta_{\tau})+E\equiv 0$.
Then it follows from Equation (\ref{cl4-10-eq-adj}) that the effective divisor $R_f' = 0$ and hence
\begin{equation*}\label{cl-4.11-eq1}
	K_X+\tau^{-1}(\Delta_{\tau})+E=f^*(K_X+\tau^{-1}(\Delta_{\tau})+E)
\end{equation*}
which implies that the restriction map
$$f|_{X\backslash (\tau^{-1}(\Delta_{\tau})\cup E)}:X\backslash (\tau^{-1}(\Delta_{\tau})\cup E)\to X\backslash (\tau^{-1}(\Delta_{\tau})\cup E)$$
 is \'etale by the purity of branch loci.
By \cite[Theorem 1.4]{MZg20}, $(X, \tau^{-1}(\Delta_{\tau})+E)$ is a toric pair.
In particular, $\tau^{-1}(\Delta_{\tau})+E$ has $\dim(X)+\rho(X)$ components.

On the other hand, both $\Delta_\tau$ and $\tau^{-1}(\Delta_\tau)$ have $\dim(Y)+\rho(Y)$ components each by Lemma \ref{lem-Y-toric} (cf.~Notation \ref{notation-conic}).
So
$\tau^{-1}(\Delta_{\tau})+E$ is a divisor with only $\dim(Y)+\rho(Y)+1=\dim(X)+\rho(X)-1$ components.
We have reached a contradiction.
\end{proof}

\begin{proof}
[{\bf End of Proof of Proposition \ref{prop-no-div}}]
By the adjunction formula, $K_X\cdot \ell=-2$ for a general fibre $\ell$ of $\tau: X \to Y$.
This and Claim \ref{cl-2-cover} imply that $K_X+\tau^{-1}(\Delta_{\tau})+E$ is $\tau$-trivial.
By the cone theorem (cf.~\cite[Theorem 3.7]{KM98}) and Claim \ref{cl-neq0},
\begin{equation*}\label{cl4-pf4.7} K_X+\tau^{-1}(\Delta_{\tau})+E\sim \tau^*D \end{equation*}
for some Cartier divisor $D\not\equiv 0$ on $Y$.
Therefore, restricting the above equation
to $E$ for both sides and applying the adjunction formula, we have
$$K_E+\tau^{-1}(\Delta_{\tau})|_E=(K_X+\tau^{-1}(\Delta_{\tau})+E)|_E\sim \tau^*D|_E=(\tau|_E)^*D\not\equiv 0,$$
where the last inequality follows from Claim \ref{cl-2-cover}.
This contradicts Claim \ref{cl-eq0}.
The proof of Proposition \ref{prop-no-div} is completed.
\end{proof}

\begin{lemma}\label{lem-branch}
Suppose that $\Delta_{\tau}\neq\emptyset$.
Suppose further that there exists an extremal contraction $\pi:X\to Z\cong \mathbb{P}^1$ of a $K_X$-negative extremal face, with reduced and irreducible fibres such that the induced map $\psi:X\to Y\times Z$ is finite surjective.
Denote by $h:=f|_Z$ after iteration.
Then the following hold:
\begin{enumerate}
\item $\deg (\psi)>1$ and $\psi$ has branch loci $B_{\psi}\neq\emptyset$.
\item Each prime divisor of $B_{\psi}$ dominates both $Y$ and $Z$.
\item $(g\times h)^*B_{\psi}=(g\times h)^{-1}(B_{\psi})$.
\item $B_{\psi}$ is $(g\times h)$-invariant.
\end{enumerate}
\end{lemma}

\begin{proof}
Let $p_Y:Y\times Z\to Y$ and $p_Z:Y\times Z\to Z$  be the projections.
If $\deg (\psi)=1$, then $X \cong Y\times Z$ and $\tau$ is a trivial (smooth) $\mathbb{P}^1$-bundle,  contradicting that $\Delta_{\tau}\neq\emptyset$.
So $\deg (\psi)>1$.
Since $Y\times Z$ is simply connected, $\psi$ has branch loci $B_{\psi}\neq \emptyset$. This proves (1).

By the purity of branch loci, write $B_{\psi}=\bigcup_{i=1}^n P_i$ with $P_i$ being prime divisors.
Let $Q$ be a prime divisor of $Y\times Z$ with  $p_Z(Q)\neq Z$.  
Then we have $Q = p_Z^*p_Z(Q)$ by the coordinate projection and hence
$\psi^*Q=\pi^*p_Z(Q)$. 
Since $p_Z(Q)$ is a single point, by our assumption, $\psi^*Q=\pi^*p_Z(Q)$ is reduced and irreducible.
Similar arguments show that $\psi^*Q$ is also a reduced and irreducible prime divisor if  $Q$ is a prime divisor of $Y\times Z$ with  $p_Y(Q)\neq Y$, noting that $\tau$ is a Fano contraction of a $K_X$-negative extremal ray (cf.~Proposition \ref{pro-cone-reduced-reducible}).
As a result, each prime divisor $P_i\subseteq B_{\psi}$ dominates both $Y$ and $Z$.
So (2) is proved.

Note that the branch locus of $g\times h$: $B_{g\times h}=p_Y^{-1}(B_g)+p_Z^{-1}(B_h)$.
By (2), no prime divisor of $B_{\psi}$ lies in the branch locus of $g\times h$.
So $(g\times h)^*P_i=(g\times h)^{-1}(P_i)=\sum_k P_i(k)$ with $P_i(k)$ being prime divisors.
In particular, (3) is proved.

By (3) and comparing the coefficients on the both sides of $\psi^*(g\times h)^*P_i=f^*\psi^*P_i$,
we may assume $P_i(1)\subseteq B_{\psi}$ for each $i$.
Note that $P_i(1)\neq P_j(1)$ for any $i\neq j$.
Then $B_{\psi}=\bigcup_{i=1}^n P_i(1)$ and hence $(g\times h)(B_{\psi})=B_{\psi}$.
So (4) is proved.
\end{proof}

\begin{lemma}\label{lem-no-p1}
Suppose that $\Delta_{\tau}\neq\emptyset$.
Suppose further that $f$ is polarized.
Then it is impossible that $X$ has an extremal contraction $\pi:X\to Z\cong \mathbb{P}^1$ of a $K_X$-negative extremal face, with reduced and irreducible fibres  such that the induced map $\psi:X\to Y\times Z$ is finite surjective.
\end{lemma}

\begin{proof}
Suppose the contrary that such $\pi$ exists.
Denote by $h:=f|_Z$ after iteration (cf.~Lemma \ref{lem-equiv-MMP}) and $p_Y:Y\times Z\to Y$, $p_Z:Y\times Z\to Z$ the two projections.

We claim that $\deg (\psi)=2$.
Note that $h$ (like $f$) is also polarized.
Let $F$ be an $f$-periodic smooth general fibre of $\pi:X\to Z$ (cf.~\cite[Theorem 5.1]{Fak03}).
After iteration, we may assume $f(F)=F$.
Consider the following commutative diagram
\[
\xymatrix{
F\ar[r]^{f|_F}\ar[d]_{\tau|_F}& F\ar[d]^{\tau|_F}\\
Y\ar[r]_g& Y
}
\]
where $\tau|_F$ is finite surjective by the finiteness of $\psi$.
Since $F$ is general, we may assume $\tau^{-1}(\Delta_{\tau})|_F=\tau^{-1}(\Delta_{\tau})\cap F$ is a reduced divisor on $F$.
By the adjunction formula, the log ramification divisor formula on $\tau|_F$ and Lemma \ref{lem-Y-toric}, we have
\begin{equation}\tag{$*$}
(K_X+\tau^{-1}(\Delta_{\tau}))|_F=K_F+\tau^{-1}(\Delta_{\tau})|_F=(\tau|_F)^*(K_Y+\Delta_{\tau})+R\sim R
\end{equation}
with $R$ an effective divisor on $F$.
Consider the log ramification divisor formula on $f|_F$:
$$K_F+\tau^{-1}(\Delta_{\tau})|_F=(f|_F)^*(K_F+\tau^{-1}(\Delta_{\tau})|_F)+R' ;$$
here $R'$ is effective  on $F$.
Hence, $-(K_F+\tau^{-1}(\Delta_{\tau})|_F)$ is pseudo-effective by Lemma \ref{lem-int-amp-equiv},
since
$f|_F$ is still polarized.
This and the $(*)$ above imply $(K_X+\tau^{-1}(\Delta_{\tau}))|_F\equiv 0$.
Hence, $K_X+\tau^{-1}(\Delta_{\tau})=\pi^*H$ for some divisor $H$ on $Z$, by the cone theorem (cf.~\cite[Theorem 3.7 (4)]{KM98}).
Write $\pi^*H\sim nF$ for some $n\in \mathbb{Z}$.
Let $\ell$ be a general fibre of $\tau$ and $\ell'\cup\ell''$ a (reduced) reducible fibre.
Since $\tau$ is a (flat) conic bundle, $\ell\equiv \ell'+\ell''$.
Then we have
$$nF\cdot \ell=\pi^*H\cdot \ell=(K_X+\tau^{-1}(\Delta_{\tau}))\cdot \ell=K_X\cdot \ell=-2$$
by the adjunction formula.
Since $F$ is integral and $\tau$-ample, $F\cdot \ell=F\cdot (\ell'+\ell'')\ge 1+1$.
So $n=-1$ and $F\cdot \ell=2$.
In particular, $\deg (\psi)=2$ and the claim is proved.

Note that $\psi$ being a double cover implies $\psi^*B_{\psi}=2\psi^{-1}(B_{\psi})$.
By Lemma \ref{lem-branch} (3), $(g\times h)^*B_{\psi}=(g\times h)^{-1}(B_{\psi})$.
Then we have:
$$2f^*\psi^{-1}(B_{\psi})=f^*\psi^*B_{\psi}=\psi^*(g\times h)^*B_{\psi}=\psi^*((g\times h)^{-1}(B_{\psi}));$$
hence $(g\times h)^{-1}(B_{\psi})\subseteq B_{\psi}$, noting that for each irreducible component $P$ of $(g\times h)^{-1}(B_\psi)$, we have $\psi^*P=2Q$ for some prime divisor $Q$ on $X$.
Since $g\times h$ is finite surjective, we have $(g\times h)^{-1}(B_{\psi})=B_{\psi}$.
Then $B_{\psi}\subseteq B_{g\times h}= p_Y^{-1}(B_g)+p_Z^{-1}(B_h)$
since $g\times h$ (like its lifting $f$) is also polarized
 (cf.~Lemma \ref{lem-pol-mod>1}).
This contradicts Lemma \ref{lem-branch} (2).
\end{proof}

\begin{lemma}\label{lem-p2-pol}
Suppose that $\Delta_{\tau}\neq\emptyset$ and $Y\cong \mathbb{P}^2$.
Then $f$ is polarized after iteration.
\end{lemma}

\begin{proof}
Note that $\rho(X)=\rho(Y)+1=2$ and $-K_X$ is ample.
Let $\pi:X\to Z$ be the second contraction of a $K_X$-negative extremal ray.
After iteration, we may assume $\pi$ is $f$-equivariant (cf.~Lemma \ref{lem-equiv-MMP}).
Note that $\rho(Z)=1$.
If $\dim(Z)=\dim(X)$, then $f|_Z$ is non-isomorphic and hence polarized; thus
$f$ is polarized (cf.~Lemma \ref{lem-pol-mod>1}).
So we may assume $\pi$ is a Fano contraction.
Then
$Z\cong \mathbb{P}^2$ or $\mathbb{P}^1$
(cf.~Lemma \ref{lem-base-rat}).
Set $h := f|_Z$.
Then $h^*\mathcal{O}_Z(1)=\mathcal{O}_Z(p)$ for some $p\ge 1$.
On the other hand, since $g=f|_Y$ has $\deg(g) > 1$ by Proposition \ref{prop-zhang} and $Y \cong \PP^2$,
we have $g^*\mathcal{O}_Y(1)=\mathcal{O}_Y(q)$ for some $q>1$.
Now $f^*|_{\N^1(X)}=\diag[q,p]$.
It then suffices to show that $p=q$ (cf.~Lemma \ref{lem-pol-mod>1}).

Consider the case $Z\cong \mathbb{P}^2$.
Let $D_1=\tau^*\mathcal{O}_{Y}(1)$ and $D_2=\pi^*\mathcal{O}_{Z}(1)$.
Then $f^*D_1\sim qD_1$ and $f^*D_2\sim pD_2$.
Since $\tau \ne \pi$, we have $D_1^2\cdot D_2>0$ and $D_2^2\cdot D_1>0$.
By the projection formula, $\deg (f)=q^2p=p^2q$.
So $q=p$ and $f$ is polarized.

Consider the case $Z\cong \mathbb{P}^1$.
Then our $\pi$ here is a Fano contraction of a $K_X$-negative extremal ray (and hence has reduced and irreducible fibres; cf.~Proposition \ref{pro-cone-reduced-reducible}).
Then the induced map $\psi:X\to Y\times Z$ is finite surjective since $\tau \ne \pi$ and
$\rho(X)=\rho(Y\times Z)$.
By Lemma \ref{lem-branch}, $\deg (\psi)>1$, the branch locus $B_\psi\neq\emptyset$ and we may assume (after iteration) $(g\times h) (P_1)=P_1$ for some prime divisor $P_1\subseteq B_\psi$.
Denote by $p_Y:Y\times Z\to Y$ and $p_Z:Y\times Z\to Z$ the two projections.
Since $p_Y|_{P_1}:P_1\to Y$ is generically finite and surjective by Lemma \ref{lem-branch} (2),
$(g\times h)|_{P_1}$ (like its descending $g$ on $Y$) is $q$-polarized by Lemma \ref{lem-pol-mod>1}.
Since $p_Z|_{P_1}:P_1\to Z$ is surjective,
$h$ on $Z$ (like its lifting $(g\times h)|_{P_1}$) is $q$-polarized by  Lemma \ref{lem-pol-mod>1}.
So $p=q$ and  $f$ is polarized.
\end{proof}

\begin{lemma}\label{lem-p1p1-pol}
Suppose that $\Delta_{\tau}\neq \emptyset$ and $Y\cong \mathbb{P}^1\times \mathbb{P}^1$.
Then $f$ is polarized after iteration.
\end{lemma}

\begin{proof}
Note that $\rho(X) = \rho(Y)+1 = 3$.
After iteration, we write $f^*|_{\N^1(X)}=\diag[\lambda_1,\lambda_2,\lambda_3]$ and  $g^*|_{\N^1(Y)}=\diag[\lambda_2,\lambda_3]$ (cf.~\cite[Lemma 9.2]{Men20} or \cite[Lemma 5.4]{MZ20}).
We may assume $\lambda_2\le \lambda_3$.
By Proposition \ref{prop-zhang}, $\deg (g)>1$, hence $\lambda_3>1$.
Denote by $p_i:Y\to Y_i\cong \mathbb{P}^1$ ($i = 2, 3$) the two projections such that $\deg (g|_{Y_i})=\lambda_i$.
By Lemma \ref{lem-simple-loop}, $g^{-1}(\Delta_{\tau})=\Delta_{\tau}$.
Hence, $\Delta_\tau$ contains at least (and indeed exactly) two fibres each of  $p_2$ and $p_3$ by Lemmas \ref{lem-p1p1-ticurve} and \ref{lem-Y-toric},
so it is ample.

We first claim that either $\lambda_2>1$ or $\lambda_1=\lambda_3>1$.
Suppose $\lambda_2=1$.
Let $L$ be a general fibre of $p_2$ such that $L\not\subseteq \Delta_{\tau}$.
Since $g^{-1}(\Delta_{\tau})=\Delta_{\tau}$, we have $g(L)\not\subseteq \Delta_{\tau}$.
Since $g|_{Y_2}$ is an automorphism, $\#L\cap \Delta_{\tau}=L\cdot \Delta_{\tau}=g(L)\cdot \Delta_{\tau}=\#g(L)\cap \Delta_{\tau}\ge 1$.
Let $F:=\tau^{-1}(L)$.
Then $F$ (resp.~$f(F)$) is a smooth ruled surface with a ruling $\tau|_F: F \to L$ (resp.~$\tau|_{f(F)}: f(F) \to g(L)$) having
$\#L\cap \Delta_{\tau}$ ($>0$) singular fibres being intersecting $(-1)$-curves.
Let $\ell$ be an irreducible component of a singular fibre of $\tau|_F$ lying over $P:=\tau(\ell) \in L$.
Since $P\in L\cap \Delta_{\tau}$, we have $g^{-1}(g(P))=P$ and $(f|_F)^*(f(\ell))=\lambda_3 \ell$ by noting that $\deg (g|_L)=\lambda_3$.
By the projection formula,
$$-\deg (f|_F)=\deg (f|_F)\cdot \ell^2=((f|_F)^*(f(\ell)))^2=\lambda_3^2\ell^2=-\lambda_3^2.$$
Note that $\deg (f|_F)=\lambda_1\lambda_3$.
Then we have $\lambda_1=\lambda_3$ and the claim is proved.

Let $L'$ be a $g$-periodic general fibre of $p_3$ such that $L'\not\subseteq \Delta_{\tau}$ (cf.~\cite[Theorem 5.1]{Fak03}).
After iteration, we may assume $g(L')=L'$ with $\deg (g|_{L'})=\lambda_2$.
Denote by $F':=\tau^{-1}(L')$.
Then $(f|_{F'})^*|_{\N^1(F')}$ has two positive integral eigenvalues $\lambda_1, \lambda_2$
with either $\lambda_1 > 1$ or $\lambda_2 > 1$ by the claim above.
In particular, $f|_{F'}$ is non-isomorphic.
Our $F'$ is a smooth ruled surface over $L'$ with $\# L'\cap \Delta_{\tau}$ ($>0$) singular fibres.
By Lemma \ref{lem-noniso-nonpol},
$f|_{F'}$ is polarized after iteration.
So $\lambda_1=\lambda_2$.
If $\lambda_1=\lambda_3$, then $f$ is polarized.

Suppose $\lambda_1 \ne \lambda_3$.
By the claim above, $\lambda_2 > 1$.
Taking a $g$-periodic fibre of $p_2$, and arguing as above, we get $\lambda_1 = \lambda_3$, a contradiction.
So $f$ is polarized as required.
\end{proof}

\begin{proposition}\label{prop-conic-p2}
Suppose that $\Delta_{\tau}\neq\emptyset$.
Then $Y\not\cong \mathbb{P}^2$.
\end{proposition}

\begin{proof}
Suppose the contrary that $Y\cong \mathbb{P}^2$.
Then $\rho(X) = \rho(Y)+1 = 2$.
By Lemma \ref{lem-p2-pol}, $f$ is polarized (after iteration).
Let $\pi:X\to Z$ be the second contraction of a $K_X$-negative extremal ray.
After iteration, we may assume $\pi$ is $f$-equivariant (cf.~Lemma \ref{lem-equiv-MMP}).
Then $h := f|_Z$ (like $f$) is also polarized.
By Proposition \ref{prop-no-div}, $\pi$ has to be a Fano contraction.
Since $\rho(Z) = \rho(X)-1 = 1$, our $Z\cong \mathbb{P}^2$ or $\mathbb{P}^1$ by Lemma \ref{lem-base-rat}.
Suppose $Z\cong \mathbb{P}^1$.
Then the induced map $\psi:X\to Y\times Z$ is finite surjective since $\tau \ne \pi$ and $\rho(X)=\rho(Y\times Z)=2$.
However, this contradicts Lemma \ref{lem-no-p1} (cf.~Proposition \ref{pro-cone-reduced-reducible}).
So we may assume $Z\cong \mathbb{P}^2$.

After iterating $f$, we may assume that every component $C_i$ of $\Delta_\tau$ is $g^{-1}$-invariant (cf.~Lemma \ref{lem-simple-loop}).
Hence, $S_i = \tau^{-1}(C_i)$ and $S_i \cap S_j$ are $f^{-1}$-invariant.
Consider the irreducible non-reduced (double line) fibre with reduced structure $\ell_d=S_1\cap S_2=\tau^{-1}(C_1\cap C_2)$, which is non-empty  (cf.~\cite[Theorem 3.5]{Mor82} and Lemma \ref{lem-simple-loop}).
Since $\ell_d$ is not contracted by $\pi$, the image $\pi(\ell_d)$ is an $h^{-1}$-invariant irreducible curve by \cite[Lemma 7.5]{CMZ20}.
Denote by $D:=\pi^{-1}(\pi(\ell_d))$ which is a prime divisor on $X$ by the cone theorem (cf.~\cite[Theorem 3.7]{KM98}).
Then $f^{-1}(D)=D$ and $D$ is different from $S_1, S_2$.
Hence, the pair $(X, S_1+S_2+D)$ is not lc at $\ell_d$ since all the three prime divisors $S_1, S_2, D$ contain $\ell_d$ (cf.~\cite[Lemma 2.29]{KM98}). This contradicts Theorem \ref{thm-bh}.
\end{proof}

\begin{proposition}\label{prop-conic-p1p1}
Suppose that $\Delta_{\tau}\neq\emptyset$.
Then $Y\not\cong \mathbb{P}^1\times \mathbb{P}^1$.
\end{proposition}

\begin{proof}
Suppose the contrary that $Y\cong \mathbb{P}^1\times \mathbb{P}^1$. Then $\rho(X) = \rho(Y) + 1 = 3$.
By Lemma \ref{lem-p1p1-pol}, $f$ is polarized (after iteration).
Let $\pi_1:X\to Z_1$ be the second contraction of a $K_X$-negative extremal ray.
After iteration, we may assume that $\pi_1$ is $f$-equivariant (cf.~Lemma \ref{lem-equiv-MMP}).
By Proposition \ref{prop-no-div}, $\pi_1$ has to be a Fano contraction and
$Z_1$ is a smooth rational (ruled) surface with $\rho(Z_1)= \rho(X) - 1 = 2$ (cf.~Lemma \ref{lem-base-rat}).

We claim that $Z_1\cong \mathbb{P}^1\times \mathbb{P}^1$.
Suppose the contrary.
Then $Z_1$ contains some negative curve $C$.
Let $E:=\pi_1^{-1}(C)$ ($= \pi_1^*(C)$ by the cone theorem).
By the projection formula, $E\cdot \widetilde{C}<0$ for any curve $\widetilde{C}\subseteq X$ dominating $C$.
Since $-K_X$ is ample, the Mori cone
$\NE(X)$ is generated by finitely many extremal (rational) curves, one of which, still denoted as $\widetilde{C}$,
is $E$-negative and is of course $K_X$-negative too.
Then the contraction of the ``non-nef" extremal ray $\R_{\ge 0}[\widetilde{C}]$ is birational and hence divisorial by \cite[Theorem 3.3]{Mor82}.
However, this contradicts Proposition \ref{prop-no-div}.
So the claim is proved.

Let $\ell$ be a general fibre of $\tau$.
Then $\ell$ is not contracted by $\pi_1$.
Let $p:Z_1\to Z\cong\mathbb{P}^1$ be one of the natural projections such that $\pi_1(\ell)$ is not contracted by $p$.
Denote by $\pi:X\xrightarrow{p\circ \pi_1} Z$ the composition.
Then the induced map $\psi:X\to Y\times Z$ is finite surjective since $\tau \ne \pi$ and
$\rho(X)=\rho(Y\times Z)=3$.
For any $t\in Z$, the fibre $p^*(t)$ is a prime divisor in $Z_1\cong\mathbb{P}^1\times\mathbb{P}^1$.
By Proposition \ref{pro-cone-reduced-reducible}, $\pi$ has reduced and irreducible fibres.

Let $H$ be an ample divisor of $Z$ and denote by
$$F:=\{C\in \NE(X)\,|\, \pi^*H\cdot C=0\}.$$
Since $\pi^*H$ is nef, $F$ is a ($K_X$-negative) extremal face of $\NE(X)$.
By the projection formula, any curve $C$ is contracted by $\pi$ if and only if $\pi^*H\cdot C=0$.
So $\pi$ is the contraction of the face $F$.
However, this contradicts Lemma \ref{lem-no-p1}.
\end{proof}

\begin{proof}
[{\bf End of Proof of Theorem \ref{thm-conic}}]
Suppose $\Delta_{\tau}\neq\emptyset$.
By Lemma \ref{lem-p2orp1p1}, $Y\cong \mathbb{P}^2$ or $\mathbb{P}^1\times \mathbb{P}^1$.
However, this contradicts Propositions \ref{prop-conic-p2} and \ref{prop-conic-p1p1}.
Thus, $\tau$ is a (smooth) $\mathbb{P}^1$-bundle over $Y$.
By Lemma \ref{lem-p1-bundle-projective}, $X$ is an algebraic $\mathbb{P}^1$-bundle $\mathbb{P}_Y(\mathcal{E})$ over $Y$.
We are done.
\end{proof}

\section{Fano threefolds of Picard number two}\label{section-5-rho2}

In this section, we deal with the case where the Picard number $\rho(X)=2$ and prove:

\begin{theorem}\label{thm-rho2}
Let $X$ be a smooth Fano threefold of Picard number $\rho(X)\le 2$.
Suppose that $X$ admits a non-isomorphic surjective endomorphism $f$.
Then $X$ is toric. To be precise, $X$ is either $\mathbb{P}^3$, or a splitting $\mathbb{P}^1$-bundle over $\mathbb{P}^2$, or a blowup of $\mathbb{P}^3$ along a line.
\end{theorem}

We prepare two lemmas before proving our main result of this section.

\begin{lemma}\label{lem-rho2-div}
Let $\pi:X\to Y\cong \mathbb{P}^3$
be a blowup of $Y$ along a smooth curve $C$.
Suppose that $X$ is Fano and admits a non-isomorphic surjective endomorphism $f$.
Then the second extremal contraction $\tau:X\to Z$ is a Fano contraction.
\end{lemma}

\begin{proof}
Note that $\rho(X) = \rho(Y) + 1 = 2$. Suppose the contrary that $\tau:X\to Z$ is birational.
Then it is a divisorial contraction by \cite[Theorems 3.3]{Mor82}.
Let $E$ and $F$ be the exceptional loci of $\pi$ and $\tau$, respectively, which are known to be prime divisors.
After iteration, we may assume $\pi$ and $\tau$ are $f$-equivariant and denote by $g:=f|_Y$ and $h:=f|_Z$ (cf.~Lemma \ref{lem-equiv-MMP}).
Note that $g$ and hence $f$ and $h$ are all polarized (cf.~Lemma \ref{lem-pol-mod>1}).

We claim that $E\neq F$.
Suppose the contrary that $E=F$.
Then $\tau(F)$ is not a point; otherwise $\pi$ and $\tau$ will contract some common curve, a contradiction.
Hence, $\tau$ is the blowup of a smooth threefold $Z$ along a smooth curve by \cite[Theorem 3.3]{Mor82}.
By \cite[Lemma 2.29]{KM98}, $K_X=\pi^*K_Y+E=\tau^*K_Z+F$ and hence $K_X-E$ is both $\pi$-trivial and $\tau$-trivial.
Since $\rho(X)=2$, we have $K_X-E\equiv 0$, a contradiction to $X$ being Fano.
The claim is proved.

By this claim, $\pi(F)$ and $\tau(E)$ are both surfaces.
If $E\cap F=\emptyset$, then by the cone theorem (cf.~\cite[Theorem 3.7]{KM98}), $F=\pi^*\pi(F)$ is nef, a contradiction to $-F$ being $\tau$-ample.
So $E\cap F\neq \emptyset$.
Let $C_0$ be an $f^{-1}$-invariant (after iteration) irreducible curve contained in $E\cap F$.
Since $\pi \ne \tau$, this $C_0$ is not contracted by  $\pi$ and $\tau$ at the same time.

Suppose $\tau(C_0)$ is not a point.
Then $\tau(C_0)=\tau(F)\subseteq \tau(E)$, and
$\tau: X \to Z$ is the blowup along the (smooth)
curve $\tau(F)$ by \cite[Theorem 3.3]{Mor82}.
Hence,
the rationally connected smooth threefold $Z$ of Picard number one has $Z\cong\mathbb{P}^3$ (cf.~Theorem \ref{thm-fano-rho=1}).
Note that $\tau(E)$ is a surface in $Z$.
By \cite[Lemma 7.5]{CMZ20}, $h^{-1}(\tau(E))=\tau(E)$.
So the $h^{-1}$-periodic $\tau(E)$
and hence $\tau(C_0) = \tau(F)$
are both linear
by Theorem \ref{thm-linear-subspace}.
Thus $\tau : X \to Z \cong \PP^3$ is a toric blowup along the line $\tau(F)$, and $X$ is toric.

Suppose $\tau(C_0)$ is a point.
Then $\pi(C_0)$ is not a point and $C=\pi(C_0)=\pi(E)\subseteq \pi(F) \subseteq Y \cong \PP^3$.
By Theorem \ref{thm-linear-subspace},
the $g^{-1}$-periodic $\pi(F)$ and hence $C$ are both linear. So $\pi: X \to Y \cong \PP^3$ is the toric blowup along the line $C$ and is toric.

Thus, in both cases, $X$ is a toric blowup of $\PP^3$ along a line $\ell$ ($=C$ or $\tau(F)$).
Then there is a free pencil $\varphi:X\to B:=\PP^1$, parameterizing hyperplanes in $\PP^3$ passing through $\ell$.
Since $\rho(X)=\rho(B)+1$ and $-K_X$ is ample,
$\varphi$ ($\neq \pi$) is a Fano contraction of a $K_X$-negative extremal ray and hence $\varphi=\tau$, a contradiction.
So our lemma holds.
\end{proof}

\begin{lemma}\label{lem-div-fanop1}
Let $\pi:X\to Y\cong \mathbb{P}^3$
be a blowup of $Y$ along a smooth curve $C$.
Suppose that $X$ is Fano and admits a non-isomorphic surjective endomorphism $f$.
Suppose further that $X$ admits a Fano contraction $\tau:X\to Z$ to a curve.
Then $C$ is a line, $\pi$ is a toric blowup and $X$ is toric.
\end{lemma}

\begin{proof}
By Lemma \ref{lem-base-rat}, $Z \cong \PP^1$.
We may assume $\pi$ and $\tau$ are $f$-equivariant after iteration and
denote by $g:=f|_Y$ and $h:=f|_Z$ (cf.~Lemma \ref{lem-equiv-MMP}).
Note that $g$ and hence $f$ and $h$ are all polarized by Lemma \ref{lem-pol-mod>1}.
Let $E$ be the ($f^{-1}$-invariant) exceptional divisor of $\pi$.
Note that $\tau|_E:E\to Z$ is surjective because $\pi$ and $\tau$ contract no common curves.

Let $F$ be a general $f$-periodic fibre of $\tau$
(cf.~\cite[Theorem 5.1]{Fak03}).
After iteration, we assume that $f(F)=F$.
Then $f|_F$ (like $f$) is also polarized.
Hence, $F$ is a smooth toric Fano surface
(cf.~Theorem \ref{thm-surface-toricpair}).
Since $E$ is $\tau$-ample,
$E|_F$ is an ample divisor on $F$; hence the support $E\cap F$ is connected (cf.~\cite[Chapter \Rmnum{3}, Corollary 7.9]{Har77}).
By the generic smoothness of $\tau|_E:E\to Z$ (cf.~\cite[Chapter \Rmnum{3}, Corollary 10.7]{Har77}), we see that $C_0:=E\cap F$ is smooth and hence irreducible.
Since $f^{-1}(E)=E$, we have $(f|_F)^{-1}(C_0)=C_0$ and thus $C_0\cong\mathbb{P}^1$ (cf.~Corollary \ref{cor-ti-rational}).
Now that $\pi \ne \tau$ and $C_0$ dominates $C$ via $\pi|_E : E \to C$, we have $C\cong\mathbb{P}^1$.
As a result, the exceptional divisor $E$, as a smooth rational surface of Picard number two admitting two rulings, has to be $\mathbb{P}^1\times\mathbb{P}^1$ with $C_0$ being a fibre of the (smooth) ruling $\tau|_E$.

Further, if $\ell$ is any fibre of $\pi|_E : \PP^1 \times \PP^1 \cong E \to C$, we have
$(C_0 \cdot \ell)_E = 1$.
Note that $C\subseteq \pi(F)$.
Write $\pi^*\pi(F)=F+eE$ for some integer $e>0$.
Then we have
$$F\cdot \ell=(F|_E\cdot \ell)_E=(C_0\cdot \ell)_E=1.$$
Since $\pi^*\pi(F)$ is $\pi$-trivial, we have
$$1=F\cdot \ell=(F+eE)\cdot \ell-eE\cdot \ell=-e E\cdot \ell$$
which implies $e=1$ and $E\cdot \ell =-1$.
Note that
$$1\le \pi(F)^3=(\pi^*\pi(F))^2\cdot F=(F+E)^2\cdot F=E^2\cdot F=(C_0^2)_F,$$
Note  also that $C_0 = E|_F$ is $(f|_F)^{-1}$-invariant.
By Corollary \ref{cor-ti-rational}, $(C_0^2)_F\le 1$.
So $\pi(F)^3=1$ and hence $\pi(F)\cong \mathbb{P}^2$ is a hyperplane in $Y\cong \mathbb{P}^3$.
Since $g|_{\pi(F)}$ (like $g$) is polarized,
the $(g|_{\pi(F)})^{-1}$-invariant curve
$C = \pi(C_0)$ is a line in $\pi(F)$
(and also in $Y$) by Theorem \ref{thm-linear-subspace}.
Thus, $\pi$ is a toric blowup along a line $C \subseteq Y \cong \PP^3$
and $X$ is toric.
\end{proof}

\begin{proof}[Proof of Theorem \ref{thm-rho2}]
If $\rho(X)=1$, then $X\cong \mathbb{P}^3$ by Theorem \ref{thm-fano-rho=1}.

Assume $\rho(X)=2$.
If there is a Fano contraction $\tau:X\to Y$ to a surface (this is the case when $X$ is primitive by
\cite[Theorem 1.6]{MM83}), then $\rho(Y) = \rho(X) -1 = 1$ implies $Y\cong \mathbb{P}^2$
(cf.~Lemma \ref{lem-base-rat}).
By Theorem \ref{thm-conic}, $X:=\mathbb{P}_Y(\mathcal{E})\xrightarrow{\tau} Y$ is an algebraic $\mathbb{P}^1$-bundle with $\mathcal{E}$ a locally free sheaf of rank $2$ over $Y \cong \PP^2$.
This $\mathcal{E}$ splits, and hence $X$ is toric (cf.~Propositions \ref{prop-ame-split} and \ref{prop-toric-pair}).

Suppose that $X$ has no Fano contraction to a surface.
Then $X$ is imprimitive, so
there is a divisorial contraction $\pi:X \to X_1$
to a smooth Fano threefold $X_1$ with $\rho(X_1) = \rho(X) - 1 = 1$, and $\pi$ is the blowup along a smooth curve $C \subseteq X_1$.
We may assume that $f$ descends to a (non-isomorphic) surjective endomorphism $f|_{X_1}$ after iteration (cf.~Lemma \ref{lem-equiv-MMP}).
By Theorem \ref{thm-fano-rho=1}, $X_1 \cong \PP^3$.
Now our theorem follows from Lemmas \ref{lem-rho2-div} and \ref{lem-div-fanop1}.
\end{proof}

\section{Conic bundle and minimal model program}\label{section-6-mmp}

In this section, we deal with general Fano threefolds.
First, we observe:

\begin{theorem}\label{thm-mm83-rho>=3}
Let $X$ be a smooth Fano threefold with $\rho(X)\ge 3$.
Suppose that $X$ admits a non-isomorphic surjective endomorphism $f$.
Then $X$ admits a conic bundle.
\end{theorem}

\begin{proof}
By \cite[Section 9]{MM83}, either $X$ admits a conic bundle $\tau:X\to Y$ or $X$ is isomorphic to the blowup of $\mathbb{P}^3$ along a disjoint union of a line and a conic (and hence of Picard number $3$).
For the latter case, let $X'$ be the blowup of $\mathbb{P}^3$ along the conic curve which is still Fano by the ramification divisor formula or by \cite[Corollary 4.6]{MM83}.
One can verify that $X'$ is not toric.
By Theorem \ref{thm-rho2}, $X'$ and hence $X$ admit no non-isomorphic surjective endomorphism, a contradiction.
So $X$ admits a conic bundle.
\end{proof}

Indeed, if a smooth Fano threefold $X$ has the Picard number $\rho(X)\ge 6$, then $X$ splits into a product of $\mathbb{P}^1$ and a del Pezzo surface (cf.~\cite[Theorem 1.2]{MM83} and \cite[Theorem 1.1]{Cas12}).
Encouraged by Theorem \ref{thm-mm83-rho>=3}, next we aim to reduce a general (Fano) conic bundle to a Fano contraction $X \to Y$ to a surface $Y$ as in Theorem \ref{thm-conic}, so that $\rho(X) = \rho(Y) + 1$.
So we apply the relative minimal model program.

\begin{theorem}\label{thm-MMP}
Let $X$ be a smooth Fano threefold with a conic bundle $\tau:X\to Y$.
Suppose that $X$ admits a non-isomorphic surjective endomorphism $f$.
Then there exists an $f$-equivariant (after iteration) minimal model program
$$X=X_1\rightarrow \cdots \rightarrow X_i\rightarrow\cdots\rightarrow X_{r+1}\rightarrow Y$$
such that the following hold.

\begin{enumerate}
\item $r=\rho(X)-\rho(Y)-1$ and each $X_i$ is a smooth Fano threefold.
\item $\tau_{r+1}:X_{r+1}\to Y$ is a Fano contraction and an algebraic $\mathbb{P}^1$-bundle $\mathbb{P}_Y(\mathcal{E})$ over $Y$.
\item The composition $\tau_i:X_i\to Y$ is a conic bundle with
$(f|_Y)^{-1}$-invariant discriminant
$\Delta_{\tau_i}=C_i\cup\cdots \cup C_r$ a disjoint union of $r+1-i$ smooth curves on $Y$.
\item The composition $\pi:X\to X_{r+1}$ is the blowup of $X_{r+1}$ along $r$ disjoint union of
$(f|_{X_{r+1}})^{-1}$-invariant
smooth curves $\bigcup_{i=1}^r\overline{C_i}$ with $\tau_{r+1}(\overline{C_i})=C_i$.
\end{enumerate}
\end{theorem}

\begin{proof}
By \cite[Propositions 6.2, 6.3, Corollary 6.4 and Propositions 6.5, 6.8]{MM83}, we may run a relative minimal model program
$X=X_1\rightarrow \cdots \rightarrow X_i\rightarrow\cdots\rightarrow X_{r+1}$
 of $X$ over $Y$ which is $f$-equivariant after iteration (cf.~Lemma \ref{lem-equiv-MMP}),
such that:

\begin{itemize}
\item[(i)] $r=\rho(X)-\rho(Y)-1$ and each $X_i$ is a smooth Fano threefold; and
\item[(ii)] The composition $\tau_i:X_i\to Y$ is a conic bundle with $\Delta_{\tau_i}=\Delta_{\tau_{r+1}}\cup C_i\cup\cdots \cup C_r$ where $C_i,\cdots, C_r$ are connected (and irreducible) components of $\Delta_{\tau_i}$.
\end{itemize}

So (1) is proved.
Note that $\rho(X_{r+1})=\rho(Y)+1$.
So the conic bundle $\tau_{r+1}:X_{r+1}\to Y$ is automatically a Fano contraction of a $K_{X_{r+1}}$-negative extremal ray.
After iteration, $f$ descends to a non-isomorphic surjective endomorphism of $X_{r+1}$.
By Theorem \ref{thm-conic}, $\tau_{r+1}$ is an algebraic $\mathbb{P}^1$-bundle $\mathbb{P}_Y(\mathcal{E})$ over $Y$.
So (2) is proved.
In particular, $\Delta_{\tau_{r+1}}=\emptyset$ and (3) is proved.
Finally, (4) follows easily from (3) by noting that each divisorial contraction is the blowup along some smooth curve which is either a smooth fibre of the conic bundle $X_i \to Y$ or a subsection over $Y$ (cf.~\cite[Proposition 6.8]{MM83}).
\end{proof}

In order to prove Theorem \ref{thm-rho>=3} and hence \ref{main-thm-toric}, we need to study other possible minimal model program which is a sequence of smooth (but not necessarily Fano) threefolds.

\begin{lemma}\label{lem-fanofibrebundle}
Let $X$ be a smooth Fano threefold with a conic bundle $\tau:X\to Y$ over a Hirzebruch surface $Y\xrightarrow{p} Z\cong \mathbb{P}^1$.
Suppose that $\tau$ factors through $X\xrightarrow{\pi}W\xrightarrow{\tau_0}Y$ with $\tau_0$ being a (smooth)  $\mathbb{P}^1$-bundle.
Suppose further that $X$ admits an int-amplified endomorphism $f$.
Then $p\circ \tau_0 : W \to Z$ is a fibre bundle such that either all fibres are $\mathbb{F}_0$ or all fibres are $\mathbb{F}_1$, where $\mathbb{F}_d$ denotes the Hirzebruch surface of degree $d$.
\end{lemma}

\begin{proof}
For $z\in Z$, let $Y_z:=p^{-1}(z)\cong \mathbb{P}^1$, $X_z:=(p\circ \tau)^{-1}(z)$ and $W_z:=(p\circ \tau_0)^{-1}(z)$.
Then $W_z$ is a smooth (rational) ruled surface of Picard number $2$.
Hence, $W_z \cong \mathbb{F}_d$
for some $d \ge 0$.
The lemma is equivalent to the claim that $W_z \cong \mathbb{F}_d$ has anti-ample canonical divisor,
i.e., $d = 0, 1$.
Indeed, once this claim is proved, the lemma then follows from the fact that $\mathbb{F}_0$ and $\mathbb{F}_1$ are not deformed to each other.

We prove this claim by induction on $r:=\rho(X)-\rho(Y)-1$.
If $r=0$, then $\pi$ is isomorphic and $W\cong X$ is Fano.
By the adjunction formula, $K_{W_z}=K_W|_{W_z}$ is anti-ample, as claimed.

Suppose $r>0$.
By \cite[Proposition 6.3 and Corollary 6.4]{MM83}, since $\tau_0$ is a (smooth) $\mathbb{P}^1$-bundle,
$\Delta_{\tau}$ is a disjoint union of $r$ smooth curves $C_i$,
and $\pi$ is the blowup of $W$ along the disjoint union $\coprod_{i=1}^r\overline{C_i}$ of $r$ many $(f|_{W})^{-1}$-invariant smooth curves $\overline{C_i}$ with $\tau_0(\overline{C_i})=C_i$.
Write $\tau^{-1}(C_i)=\tau^*C_i=E_i+F_i$ where $E_i$ is the $\pi$-exceptional divisor with center $\overline{C_i}$ and $F_i$ is the strict transform of $\tau_0^{-1}(C_i)$.
Denote by $L_i:=E_i\cap F_i$ the transversal intersection.
Note that $\pi|_{F_i}:F_i\to \tau_0^{-1}(C_i)$ and $\pi|_{L_i}: L_i\to \overline{C_i}$ are isomorphic.
Let $z\in Z$.

If $X_z$ is irreducible, then $X_z$ is the blowup of $W_z$ along points $\coprod_{i=1}^r\overline{C_i}\cap W_z$.
In particular, $X_z$ is smooth.
By the adjunction formula, $K_{X_z}=K_X|_{X_z}$ is anti-ample, so is $K_{W_z}$ on the image $W_z$
of $X_z$, as claimed.

If $X_z$ ($=\tau^{-1}(Y_z))$ is reducible, then $Y_z$ equals some $C_i$ (cf.~\cite[Proposition  6.3]{MM83}) and $F_i \cong W_z$ via $\pi: X \to W$.
Let $X\to X'$ be the contraction of $E_i$ over $W$.
If $E_i\not\cong \mathbb{F}_0$, then $X'$ is Fano (and smooth) by \cite[Proposition 6.5]{MM83}.
Replacing $X$ by $X'$, we are done by induction.
So we may assume $E_i\cong \mathbb{F}_0$.
Note that $E_i, F_i$ and hence $L_i$ are $f^{-1}$-invariant after iteration (cf.~Lemmas  \ref{lem-equiv-MMP} and \ref{lem-simple-loop}).
By Lemma \ref{lem-p1p1-ticurve}, we have
$$F_i\cdot L_i=(F_i|_{E_i}\cdot L_i)_{E_i}=(L_i^2)_{E_i}=0$$
since $f|_{E_i}$ (like $f$) is int-amplified.
This implies (working over $F_i$):
$$(L_i^2)_{F_i}=E_i\cdot L_i=(\tau^*C_i-F_i)\cdot L_i=C_i^2=Y_z^2=0 ,$$
so $W_z\cong F_i\cong \mathbb{F}_0$, because $L_i \subseteq F_i$ is horizontal to (indeed a cross-section of) the ruling
$F_i \to C_i$.
This proves the claim and also the lemma.
\end{proof}

Theorem \ref{thm-conic-split-p1p1-blp2} below
replaces the semi-stability condition in Theorem \ref{thm-amerik2} by the birational dominance of
a Fano threefold. It is used crucially in proving Theorem \ref{thm-rho>=3}.

\begin{theorem}\label{thm-conic-split-p1p1-blp2}
Let $X$ be a smooth Fano threefold admitting an int-amplified endomorphism $f$.
Suppose that $X$ admits a conic bundle $\tau:X\to Y$
(with $Y$ smooth), which factors as $X\xrightarrow{\pi}W\xrightarrow{\tau_0}Y$ with $\tau_0$ being a (smooth) $\mathbb{P}^1$-bundle.
Then $\tau_0$ is a splitting $\mathbb{P}^1$-bundle.	
\end{theorem}

\begin{proof}
After iteration, we may assume that $\tau$ is $f$-equivariant by Lemma \ref{lem-equiv-MMP}.
Let $g:=f|_W$ and $h:=f|_Y$, which are both int-amplified (cf.~Lemma \ref{lem-int-amp-equiv}).
Note that $Y$ is a smooth rational surface (cf.~Proof of Lemma \ref{lem-base-rat}).
By Lemma \ref{lem-p1-bundle-projective}, we may write $W=\mathbb{P}_Y(\mathcal{E})$ with $\mathcal{E}$ a rank-two vector bundle on $Y$.

If $\rho(Y)=1$, then $Y\cong\mathbb{P}^2$, and the theorem follows from Proposition \ref{prop-ame-split}.
If $\rho(Y)\ge 3$, then $\tau=\tau_0$ is a trivial bundle by \cite[Proposition 9.10]{MM83}.

Thus we may assume $\rho(Y)=2$.
Then $Y \cong \mathbb{F}_{d}$ (Hirzebruch surface).
By \cite[Corollary 6.7]{MM83}, $d = 0, 1$.
Fix a ruling $p:Y\to Z\cong\PP^1$.
For $z\in Z$, let $Y_z:=p^{-1}(z)$, $X_z:=(p\circ \tau)^{-1}(z)$ and $W_z:=(p\circ \tau_0)^{-1}(z)$.
Let $\ell$ be the fibre class of $p$ and $C_0$ the section class of $p$ which is extremal in $\NE(Y)$
and $C_0^2 = -d$.
Twisting with a suitable line bundle of $Y$, we may assume that $c_1(\mathcal{E})=aC_0+b\ell$ with $-1\le a, b\le 0$.
Then $\mathcal{O}_{Y_z}(c_1(\mathcal{E}|_{Y_z}))\cong \mathcal{O}_{Y_z}(a)$ for any $z\in Z$.
This, $a \in \{0, -1\}$, and
$\mathbb{P}(\mathcal{E}|_{Y_z}) = W_z \cong \mathbb{F}_{c}$ ($c \le 1$, cf.~Lemma \ref{lem-fanofibrebundle}) by  base change, imply
$\mathcal{E}|_{Y_z}\cong\mathcal{O}_{Y_z}\oplus\mathcal{O}_{Y_z}(a)$ for any $z\in Z$,
and $c= -a$.

Note that our $a$ here is independent of the choice of $z$;  hence the function $z\mapsto h^0(Y_z, \mathcal{E}|_{Y_z})$ is a constant function.
So the natural morphism $\varphi: p^*p_*\mathcal{E}\to\mathcal{E}$ has domain a locally free sheaf, which is an evaluation map on every fibre (cf.~\cite[Ch III, Corollary 12.9]{Har77}).
Note that the global sections of $\mathcal{E}|_{Y_z}$ are constant, since $a \le 0$.
Then we have an exact sequence
$$(*) \hskip 2pc
0\to p^*p_*\mathcal{E}\xrightarrow{\varphi}\mathcal{E} \to \mathcal{Q} \to 0$$
where $\mathcal{Q}$ is still a vector bundle.
If $a=0$, then $p_*\mathcal{E}$ has rank two (over $\PP^1$) and splits always,
and this implies $\mathcal{E}\cong p^*p_*\mathcal{E}$ splits.

We still have to consider the case $a=-1$.
Then $p_*\mathcal{E}$ is a line bundle, say $\mathcal{O}_Z(e)$ for some $e\in \mathbb{Z}$.
Then $p^*p_*\mathcal{E}\cong \mathcal{O}_Y(e\ell)$ and $\mathcal{Q}\cong \mathcal{O}_Y(-C_0+(b-e)\ell)$.
Note that
$$\textup{Ext}^1(\mathcal{Q},p^*p_*\mathcal{E})=H^1(Y,C_0+(2e-b)\ell)
= H^1(Y, K_Y + (C_0-K_Y) + (2e-b) \ell)
$$
vanishes if $2e-b\ge 0$ by the Kodaira vanishing theorem and the nef and bigness of $(C_0-K_Y) + (2e-b)\ell$
when $Y = \F_{d}$ ($d = 0, 1$).
So $2e-b\ge 0$ implies that $(*)$ and hence $\mathcal{E}$ split.
Therefore, we may assume that $2e-b<0$.

Denote by $\xi:=2c_1(p^*p_*\mathcal{E})-c_1(\mathcal{E})\sim C_0+(2e-b)\ell$.
Then neither $\xi$ nor $-\xi$ is pseudo-effective.
Hence, $\xi$ defines a non-empty wall (cf.~\cite[Chapter I, Definition 2.1.5 and Chapter II, Definition 1.2.1]{Qin93})
$$W:=\{A\in \Amp(Y)\,|\,A\cdot \xi=0\}$$
where $\Amp(Y)$ is the ample cone in $\NS_{\R}(Y)$.
By \cite[Chapter II, Theorem 1.2.3, Page 406]{Qin93}, either $\mathcal{E}$ splits or $\mathcal{E}$ is $H$-stable for (any) ample divisor $H\in \mathcal{C}$ where $\mathcal{C}$ is a chamber (surrounded by walls) near $W$ (i.e., $\overline{\mathcal{C}}\cap W\neq\emptyset$) and $A\cdot \xi<0$ for any $A\in\mathcal{C}$.
If $\mathcal{E}$ is $H$-stable, then by our dynamical assumption and Theorem \ref{thm-amerik2},
$\mathcal{E}$ still splits, a contradiction.
This proves the theorem.
\end{proof}

\section{Proof of Theorem \ref{main-thm-prod}}\label{section-7-split}

In this section, we prove Theorem \ref{main-thm-prod}.
The key step is the following:

\begin{theorem}\label{thm-pol-split}
Let $f$ be a non-isomorphic surjective endomorphism of a smooth Fano threefold $X$.
Then either $f$ is polarized after iteration or $X\cong Y\times \mathbb{P}^1$ for a del Pezzo surface $Y$.
\end{theorem}

Before proving Theorem \ref{thm-pol-split}, we need the following:

\begin{lemma}\label{key-lem-thm-5.1}
Theorem \ref{thm-pol-split} holds for the case $\rho(X)=3$.	
\end{lemma}

\begin{proof}
Note that $\Nef(X)$ is a rational polyhedron with $r$ extremal rays.
Since $\Nef(X)$ spans $\N^1(X)$, we have $r\ge\rho(X)=3$.
After iteration, we may assume
$$f^*|_{\N^1(X)}=\diag[\lambda_1, \lambda_2, \lambda_3]$$
is diagonal with $\lambda_i$ being positive integers.
We may assume that $f$ is not polarized even after iteration.
Then $\#\{\lambda_1, \lambda_2, \lambda_3\} = 2, 3$ (cf.~Lemma \ref{lem-pol-mod>1}).

We show that $r=3$.
Suppose $r\ge 4$.
Note that any three of the extremal rays are linearly independent.
If $f^*|_{\N^1(X)}$ has three distinct eigenvalues, then there are exactly three linearly independent one-dimensional eigenspaces.
However, one of them contains two extremal rays, a contradiction.
If $f^*|_{\N^1(X)}$ has two distinct eigenvalues, then there is a two-dimensional eigenspace containing three extremal rays, a contradiction.
So $r=3$ and $\Nef(X)$ (and hence $\N^1(X)$) is generated by three extremal nef divisors $D_1, D_2, D_3$.

We may assume $f^*D_i=\lambda_i D_i$.
Since $f$ is non-isomorphic and non-polarized, $D_i$ is not big
by Lemma \ref{lem-pol-mod>1}, so
$D_i^3=0$ for all $i$.
Since $-K_X$ is ample, there are three extremal ray contractions $\pi_i:X\to Y_i$ such that $\pi_i^*\Nef(Y_i)$ is generated by $\{D_1, D_2, D_3\}\backslash \{D_i\}$ and $\rho(Y_i) = 2$.
Denote by
$f_i:=f|_{Y_i}$.

Consider the case where every $\pi_i$ is a divisorial contraction.
Then each exceptional locus
of $\pi_i$ consists of a single prime divisor $E_i$
 by \cite[Theorem 3.3]{Mor82}.
Note that each $E_i$ is $\pi_i$-anti ample and hence a non-nef (effective) eigenvector of $f^*|_{\N^1(X)}$ different from $D_1, D_2, D_3$.
Thus $\#\{\lambda_1, \lambda_2, \lambda_3\} = 2$.
We may assume $\lambda_1=\lambda_2\neq \lambda_3$.
Note that $f_3$ (like its birational lifting $f$) is non-isomorphic.
Then $f_3$ is $\lambda_1$-polarized and hence so is $f$ (cf.~Lemma \ref{lem-pol-mod>1}),
a contradiction.

Thus we may assume that $\pi_1: X \to Y_1$ is a Fano contraction.
Since $\rho(Y_1)=2$, we have $Y_1\cong \F_d$, a Hirzebruch surface (cf.~Lemma \ref{lem-base-rat}).
By Theorem \ref{thm-conic}, $\pi_1$ is further an algebraic $\mathbb{P}^1$-bundle $\mathbb{P}_{Y_1}(\mathcal{E})$.
By Lemma \ref{lem-noniso-nonpol},  $\lambda_2 = \lambda_3$ if
$Y_1 \not\cong\mathbb{P}^1\times \mathbb{P}^1$.
Since $f$ is not polarized, either $\lambda_1\neq\lambda_2$ or $\lambda_1\neq\lambda_3$.
Without loss of generality, we may assume $\lambda_1\neq\lambda_3$.
Since $Y_1 \cong \F_d$, we can choose a suitable ruling $p:Y_1\to C\cong \mathbb{P}^1$ such that
$\deg (f_1|_C)=\lambda_2$.
By Lefschetz fixed point formula on $C$, there exists an $f_1$-invariant fibre $L\cong \mathbb{P}^1$ of $p$.
Denote by $F:=\pi_1^{-1}(L)$ which is $f$-invariant.
Note that $(f|_F)^*|_{\N^1(F)}$ has two distinct eigenvalues $\lambda_1, \lambda_3$ by our assumption.
Then $F\cong \F_0 = \mathbb{P}^1\times \mathbb{P}^1$ by Lemma \ref{lem-noniso-nonpol}.
Since $\pi_1$ is a (smooth) $\PP^1$-bundle and $X$ is Fano,
each fibre of $p\circ\pi_1$ is isomorphic to $\F_0$
(cf.~e.g.
\cite[Lemma 9.4]{MM83}).
By \cite[Lemma 9.3]{MM83}, $X\cong Z\times_{\mathbb{P}^1} Y_1$ for a $\mathbb{P}^1$-bundle $Z$ over $\mathbb{P}^1$.
In particular, the induced map $X\to Z$ is another Fano contraction which is either $\pi_2$ or $\pi_3$,
i.e., $Z = Y_j$ for $j = 2$ or $3$.
If $Y_1\cong \mathbb{P}^1\times\mathbb{P}^1$, then $X$ is a trivial  $\mathbb{P}^1$-bundle over $Z$ by the base change.
If $Y_1\not \cong \mathbb{P}^1\times\mathbb{P}^1$, then $\lambda_2=\lambda_3\neq \lambda_1$
implies that $f_2^*|_{\N^1(Y_2)}$ and $f_3^*|_{\N^1(Y_3)}$ have two distinct positive eigenvalues.
Then $Z = Y_j \cong \mathbb{P}^1\times \mathbb{P}^1$ by Lemma \ref{lem-noniso-nonpol} and hence $X$ is a trivial  $\mathbb{P}^1$-bundle over $Y_1$ by the base change.
This proves the lemma.
\end{proof}

\begin{proof}[Proof of Theorem \ref{thm-pol-split}]
We divide the proof into three cases.

\par \vskip 0.3pc \noindent
\textbf{Case $\rho(X)=1$.}
In this case, $f$ is polarized.

\par \vskip 0.3pc \noindent
\textbf{Case $\rho(X)=2$.}
If $X$ has a divisorial contraction to a projective variety $X_1$ which is $f$-equivariant after iteration (cf.~Lemma \ref{lem-equiv-MMP}), then
$\rho(X_1) = \rho(X) - 1 = 1$, and hence
$f|_{X_1}$ is polarized; thus so is $f$ (cf.~Lemma \ref{lem-pol-mod>1}).
So we may assume $X$ has only Fano contractions.
Thus $X$ is primitive.
By \cite[Theorem 1.6]{MM83},
there is an ($f$-equivariant, after iterating $f$) Fano contraction $\tau:X\to Y$ to a surface $Y$.
Since $\rho(Y) = \rho(X) - 1 = 1$,
we have $Y \cong \mathbb{P}^2$ by Lemma \ref{lem-base-rat}.
Set $g:=f|_Y$.
By Theorem \ref{thm-conic} and Proposition \ref{prop-ame-split}, $\tau$ is a splitting $\mathbb{P}^1$-bundle.

We may assume $\tau$ is not trivial; otherwise, $X$ splits.
We aim to prove $f$ is polarized after iteration.
After suitable twisting, we may write $X\cong\mathbb{P}_Y(\mathcal{O}_Y\oplus\mathcal{L})$ with the line bundle $\mathcal{L}$ being  not pseudo-effective on $Y$.
Then there exists a section $S$ of $\tau$ such that $\mathcal{O}_X(S)|_S\cong (\tau|_S)^*\mathcal{L}$ is not pseudo-effective.
By Lemma \ref{lem-non-pe-ti}, $f^{-1}(S)=S$ after iteration.
Write $f^*S=aS$ for some integer $a\ge 1$ and $g^*|_{\N^1(Y)}=q\id$ for some integer $q\ge 1$.
Since $\N^1(X)$ is spanned by $S$ and $\tau^*\N^1(Y)\cong \R$,
we have $f^*|_{\N^1(X)}=\diag[a,q]$.
Note  that $(f|_S)^*(S|_S)=aS|_S$ and $S|_S\not\equiv 0$.
Then $a$ is an eigenvalue of $(f|_S)^*|_{\N^1(S)}$, which implies $a=q$ is an eigenvalue of $g^*|_{\N^1(Y)}$ since $S$ is $f$-equivariantly isomorphic to $Y$.
Since $f$ is non-isomorphic, we have $a=q>1$ and hence $f$ is polarized (cf.~Lemma \ref{lem-pol-mod>1}).

\par \vskip 0.3pc \noindent
\textbf{Case $\rho(X)\ge 3$.}
By Theorems
\ref{thm-mm83-rho>=3} and
\ref{thm-MMP}, there is an $f$-equivariant (after iteration) conic bundle $\tau:X\to Y$ and denote by $g:=f|_Y$ (cf.~Lemma \ref{lem-equiv-MMP}).
In the following, we apply Theorem \ref{thm-MMP} and use the notations there.

If $\rho(X)=3$, then Theorem \ref{thm-pol-split} follows from Lemma \ref{key-lem-thm-5.1}.

Thus we may assume $\rho(X)\ge 4$.
If $\tau : X \to Y$ has the base $\rho(Y) \ge 3$, then it is a trivial conic bundle
by \cite[Proposition 9.10]{MM83}, and we are done.
So we assume $\rho(Y) \le 2$.
Hence,
$Y = \PP^2$ or $Y = \F_d$ (cf.~Lemma \ref{lem-base-rat}) with $d \le 1$ (cf.~\cite[Corollary 6.7]{MM83}).
Since the $g^{-1}$-invariant divisor $\Delta_\tau$ is a disjoint union of $r = \rho(X) - \rho(Y) -1$ of $\PP^1$'s (cf.~Theorem \ref{thm-MMP}),
we may assume
$Y\cong \F_1$ (cf.~\cite[(9.5), (9.6)]{MM83},
Lemma \ref{lem-p1p1-ticurve}).

Thus $\Delta_\tau$ is a subset of the disjoint union of the $(-1)$-curve $C_0$ on $Y \cong \F_1$ and another section $C_1$ disjoint from $C_0$
(cf.~\cite[Corollary 6.7]{MM83}).
Let $L \cong \PP^1$ be a $g$-invariant fibre along the ruling $\sigma:Y\to Z\cong\mathbb{P}^1$ by the Lefschetz fixed point formula for $f|_Z$.
Let $F:=\tau^{-1}(L)$. Then $F$ is an $f$-invariant smooth projective rational surface, which has some reducible fibres lying over $L\cap\Delta_\tau$.
In particular, $F\not\cong\mathbb{P}^1\times \mathbb{P}^1$.
If $f|_F$ is isomorphic, then
$f|_Z$ and hence $g=f|_Y$ are non-isomorphic,
so $g$ is polarized after iteration (cf.~Lemma \ref{lem-noniso-nonpol}); thus $g|_L$ is polarized and non-isomorphic, which implies $f|_F$ is also non-isomorphic, a contradiction.
Hence, $f|_F$ is non-isomorphic and thus $q$-polarized for some $q>1$, since $F\not\cong\PP^1\times\PP^1$ (cf.~Lemma \ref{lem-noniso-nonpol}).
As a result, $g|_L$ (and hence $g$) are both $q$-polarized.
Moreover, we have $f^*F\sim qF$.

Finally, we prove $f$  is $q$-polarized after iteration.
By Lemma \ref{lem-pol-mod>1}, it suffices to show $f_{r+1}$ is $q$-polarized (cf.~the notations in Theorem \ref{thm-MMP}).
Note that $\rho(X_{r+1})=\rho(Y)+1=3$ and $\tau_{r+1}$ is an algebraic $\PP^1$-bundle over $Y$.
Write $g^*|_{\N^1(Y)}=\diag[q,q]$ with $q>1$ and $f_{r+1}^*|_{\N^1(X_{r+1})/\tau_{r+1}^*\N^1(Y)}=a\id$ for some integer $a\ge 1$.
Then $f_{r+1}^*|_{\N^1(X_{r+1})}=\diag[a,q,q]$ and it suffices to show that $a=q$ (cf.~Lemma \ref{lem-pol-mod>1}).
By the projection formula, $\deg (f_{r+1})=aq^2$.
On the other hand, $\deg (f|_F)\cdot F=f_*F=(\deg (f)/q) F$ by the projection formula.
Since $f|_F$ is $q$-polarized,
we have $\deg (f_{r+1})=\deg (f)=q^3$ and hence $a=q$ as desired.
This completes the proof of our theorem.
\end{proof}

\begin{proof}[Proof of Theorem \ref{main-thm-prod}]
Clearly, (3) implies (2).
 By Theorem \ref{thm-pol-split}, (2) implies (1).
If $X\cong\mathbb{P}^1\times S$ for a del Pezzo surface $S$, then taking $f_1:\mathbb{P}^1\to\mathbb{P}^1$ to be the square map $[x:y]\mapsto [x^2:y^2]$ and $f:=f_1\times \textup{id}_S$ the product map, we have $f$ is non-isomorphic and non-int-amplified.
So (1) implies (3).
\end{proof}

\section
{Proofs of Theorem \ref{main-thm-toric} %
and Corollary \ref{cor-2}
}\label{section-8-mainproof}

To prove
Theorem \ref{main-thm-toric},  
we begin with the following, using results in Section \ref{section-6-mmp}.

\begin{theorem}\label{thm-rho>=3}
Let $X$ be a smooth Fano threefold with $\rho(X)\ge 3$.
Suppose that $X$ admits an int-amplified endomorphism $f$.
Then $X$ is toric.
\end{theorem}

\begin{proof}
We apply Theorems \ref{thm-mm83-rho>=3} and \ref{thm-MMP} and employ the same  notations there.
Especially,
$X_{r+1} = \mathbb{P}_Y(\mathcal{E}) \xrightarrow{\tau_{r+1}} Y$ is a splitting $\PP^1$-bundle (cf.~Theorem \ref{thm-conic-split-p1p1-blp2}).
Since there are only finitely many $f^{-1}$-periodic subvarieties by \cite[Corollary 3.8]{MZ20}, we may assume that they are all $f^{-1}$-invariant (after iterating $f$).
If $r=0$,  then $\tau:X\to Y$ is a splitting $\mathbb{P}^1$-bundle  and hence $X$ is toric by Theorem \ref{thm-boundary-p1bundle}.
We now assume $1 \le r$ ($= \rho(X) - \rho(Y)-1$).

Then, by Theorem \ref{thm-MMP}, $\tau$ is a conic bundle and $\Delta_\tau = \coprod_{i=1}^r C_i$ with
$C_i \cong \PP^1$.
By \cite[Proposition 6.3]{MM83}, $\tau^{-1}(C_i)=E_i\cup F_i$, and
$E_i$ and $F_i$
are (smooth) $\PP^1$-bundles over $C_i$, with the $f^{-1}$-invariant curve $E_i\cap F_i$ dominating $C_i$ (and indeed, a cross-section).

We claim that there is an $f^{-1}$-invariant section $S$ of $\tau$ dominating $Y$.
Recall the birational contraction $\pi:X\to X_{r+1}$ and the extremal Fano contraction $\tau_{r+1}:X_{r+1}\to Y$.
Note that $\pi(E_1\cap F_1)$ is an $f_{r+1}^{-1}$-invariant subsection over $Y$ (not contracted by $\tau_{r+1}$)
and $\pi(E_1\cup F_1)=\tau_{r+1}^{-1}(C_1)$ is a $\PP^1$-bundle over $C_1$.
If $\tau_{r+1}$ is a trivial bundle so that $X_{r+1} = Y \times Z\cong Y\times\PP^1$,
then its second projection induces a second ruling on $\pi(E_1\cup F_1)$; thus $\pi(E_1\cup F_1) \cong \PP^1 \times \PP^1$; now $\pi(E_1\cap F_1)$ is contained in an $f_{r+1}^{-1}$-invariant horizontal section $S_{r+1}$ of $\tau_{r+1}$ (cf.~\cite[Lemma 7.5]{CMZ20} and Lemma \ref{lem-p1p1-ticurve} applied to $\pi(E_1\cup F_1)$).
If $\tau_{r+1}$ is not a trivial bundle, as in the proof of Theorem \ref{thm-boundary-p1bundle}, some section $S_{r+1}$ of $\tau_{r+1}$ has
$S_{r+1}|_{S_{r+1}}$ not being pseudo-effective and is hence $f_{r+1}^{-1}$-invariant after iteration by Lemma \ref{lem-non-pe-ti}.
In both cases, we take $S \subseteq X$ to be the proper transform of $S_{r+1}$.
Note also that $\pi|_S:S\cong S_{r+1}$  by the observation of $\pi$ as in Theorem \ref{thm-MMP}.
So the claim is proved.

Since $\pi(S)=Y$,
we have $(E_i\cup F_i)\cap S\neq\emptyset$.
Hence, we may assume $S\cap E_i\neq \emptyset$ for each $i$.
By \cite[Proposition 6.3]{MM83} and Lemma \ref{lem-equiv-MMP}, after iteration, there is an $f$-equivariant birational morphism $\pi':X\to X'$ over $Y$ contracting all $E_i$ with $f':=f|_{X'}$,
such that:

\begin{itemize}
\item[(i)] the induced morphism $\tau':X'\to Y$ is a conic bundle with $\Delta_{\tau'}=\emptyset$; and
\item[(ii)] $\pi'$ is the blowup of $X'$ along $r$ smooth curves $C'_i:=F'_i\cap S'$, where $F'_i:={\tau'}^{-1}(C_i)$ and $S':=\pi'(S)$ are ${f'}^{-1}$-invariant prime divisors.
\end{itemize}

Our $\tau'$ is a  $\mathbb{P}^1$-bundle, but $X'$ may not be Fano.
By  Theorem \ref{thm-conic-split-p1p1-blp2},
$\tau'$ is a splitting $\PP^1$-bundle. Hence, by Theorem  \ref{thm-boundary-p1bundle}, there is a toric pair $(X', \Delta')$ such that $\Delta'$ contains all $f'^{-1}$-invariant prime divisors (including $F_i'$, $S'$).
By the construction, $\pi'$ is the composition of  toric blowups
of the intersection of two prime divisors in the log smooth toric boundary starting from the (log smooth)
toric pair $(X', \Delta')$.
Thus $X$ is toric.
\end{proof}

\begin{proof}[Proof of Theorem \ref{main-thm-toric}]
Clearly, (2) implies (3).
By Theorems \ref{thm-rho2} and \ref{thm-rho>=3}, (3) implies (1).
Note that every projective toric variety has a polarized endomorphism; see \cite[Lemma 4]{Nak02} and
\cite[Proof of Theorem 1.4]{MZg20}.
So (1) implies (2).
\end{proof}


\begin{proof}[Proof of Corollary \ref{cor-2}]
By Main Theorem, we may assume $X = \PP^1 \times S$
with $S$ a del Pezzo surface.
Iterating $f$, we may assume $f$ fixes all
(finitely many and automatically $K_X$-negative)
extremal rays of $\NE(X)$. So $f$ descends to $f_P$ on $\PP^1$ and $f_S$ on $S$ via the two projections which are contractions of extremal faces of $\NE(X)$, and $f = f_P \times f_S$.
If $\deg(f_S) \ge 2$, then $S$ and hence $X$ are toric by \cite[Theorem 3]{Nak02}.
We are done.
\end{proof}

\begin{ack}
The authors are supported by  a Research Fellowship of KIAS,  an ARF of NUS and a President's Scholarship of NUS, respectively.
The authors would like to thank the referee for very careful reading and suggestions to improve the paper.
On behalf of all authors, the corresponding author states that there is no conflict of interest.
\end{ack}

\end{document}